\documentclass[11pt]{amsart}
\usepackage{amsmath,amsthm,amssymb,amsfonts, bm, graphicx}
\usepackage[applemac]{inputenc}
\usepackage[pdftex]{hyperref}
\numberwithin{equation}{section}
\makeindex
\newcommand{\cc}{\mathfrak{c}}

\newcommand{\p}{\mathfrak{p}}

\renewcommand{\a}{\mathfrak{a}}
\newcommand{\h}{\mathfrak{h}}
\renewcommand{\l}{\mathfrak{l}}
\newcommand{\q}{\mathfrak{q}}
\newcommand{\C}{\mathbb{C}}

\newcommand{\GL}{\mathrm{GL}}

\newcommand{\R}{\mathbb{R}}

\newcommand{\Z}{\mathbb{Z}}
\newcommand{\Q}{\mathbb{Q}}
\newcommand{\A}{\mathbb{A}}
\newcommand{\I}{\mathbb{I}}

\DeclareFontFamily{OT1}{rsfs}{}
\DeclareFontShape{OT1}{rsfs}{n}{it}{<-> rsfs10}{}
\DeclareMathAlphabet{\mathscr}{OT1}{rsfs}{n}{it}

\newtheorem{lemma}{Lemma}
\newtheorem{thm}{Theorem}
\newtheorem{prop}{Proposition}

\newtheorem{cor}{Corollary}

\theoremstyle{definition}

\newtheorem{rem}{Remark}

\begin{document}

\author{Valentin Blomer}
\address{Universit\"at G\"ottingen, Mathematisches Institut,  Bunsenstr. 3-5, 37073 G\"ottingen} \email{blomer@uni-math.gwdg.de}

\author{Farrell Brumley}
\address{ Institut \'Elie Cartan, Universit\'e Henri Poincar\'e Nancy 1, BP 239, 54506 Vand{\oe}uvre Cedex, France} \email{farrell.brumley@iecn.u-nancy.fr}
\title[Families and non-vanishing]{Non-vanishing of $L$-functions, the Ramanujan conjecture, and families of Hecke characters}

\thanks{The first author is supported by a Volks\-wagen Lichtenberg Fellowship and a European Research Council (ERC) starting grant 258713. The second author is supported by the ANR grant Modunombres.}


\begin{abstract}  We prove a non-vanishing result for families of $\GL_n\times\GL_n$ Rankin-Selberg $L$-functions in the critical strip, as one factor runs over twists by Hecke characters.  As an application, we simplify the proof, due to Luo, Rudnick, and Sarnak, of the best known bounds towards the Generalized Ramanujan Conjecture at the infinite places for cusp forms on $\GL_n$.  A key ingredient is the regularization   of the units in residue classes by the use of an Arakelov ray class group. 
\end{abstract}


\maketitle

\section{Introduction}

In this paper we prove a non-vanishing result for Rankin-Selberg $L$-functions for cusp forms on $\GL_n$ when one factor ranges over twists by infinite order Hecke characters and give applications to bounds towards the Ramunujan conjecture.

\subsection{Infinite order Hecke characters}\label{inf ord} Before stating our precise results, we enunciate two ways, one automorphic, the other arithmetic, in which infinite order Hecke characters have been put to use to solve problems in number theory whose formulation is naturally posed over number fields.

The first is the principle that a cusp form (for the group $\GL_2$ at least) should always be taken together with its twists by mod 1 Grossencharacters.  This philosophy has been attributed to Sarnak and has its origins in the hybrid mean value estimate of \cite{S85}.  A striking justification is the recent article by Booker and Krishnamurthy \cite{BK} where it is shown that for an irreducible two dimensional complex representation $\rho$ of the Weil-Deligne group, if the Artin-Weil $L$-function $L(s,\rho\otimes\omega)$ is entire for all mod 1 Grossencharacters $\omega$ then $\rho$ is automorphic.

The second viewpoint, which is the one we shall emphasize in this paper, is that infinite order Hecke characters allow one to regularize the behavior of an erratic, arithmetically defined quantity.  The simplest example is the class number $h_K$ of a number field $K$, measured with respect to the discriminent.   Although there exist constructions of sparse sequences of $K$ for which $h_K$ is as large as possible (see  \cite{MW} for a detailed look at the classical example of $K=\Q(\sqrt{n^2+1})$, and Duke \cite{Du} for a generalization to certain cubic fields), within {\it arbitrary} sequences of number fields with a fixed signature, the class number $h_K$ varies wildly.  Recent experience has shown, however, (see \cite{EV} and \cite[Section 1.4]{MV}) that by making full use of the infinite places, i.e., using Arakelov structures, one can bypass the irregular behavior of $h_K$.

A more pertinent example to our line of inquiry is the size of the ray class group mod $\q$ for a {\it fixed} number field $K$ and varying integral ideals $\q$.  The size of the ray class group is determined by the size of the image of the units of $K$ within the group of invertible residue classes mod $\q$.  When $K$ has at least two archimedean places, the unit group is infinite, and its image can be as large as possible, conjecturally for a sequence of prime ideals $\q$ of positive density (\cite{CW, L}).  On the other hand, similarly to the specially constructed number fields in the previous example, there exists a sparse subsequence of ideals $\q$ for which one can prove that the ray class group is large.  This is the fundamental work of Rohrlich \cite{R}, in which he shows that for every $\varepsilon>0$ there are infinitely many square-free ideals $\q$ such that the size of the ray class group is at least $\mathcal{N}(\q)^{1-\varepsilon}$. This result  was of vital importance not only to Rohrlich's own work but also to several later papers concerning non-vanishing over number fields (\cite{BR}, \cite{LRS2}). Note, however, that the sparsity of the $\q$ constructed by Rohrlich becomes a liability when one wants to average over moduli, as in the work of Kim-Sarnak \cite{KimSa}.  As above, this problem can be resolved taking advantage of the infinite places, replacing the ray class group with something larger supporting infinite order characters.

We were led to consider these constructions in algebraic number theory as a result of our own work on extending the bounds towards the Ramanujan conjecture of Kim-Sarnak to arbitrary number fields.  This paper represents a fully expanded version of our coda in \cite{BB}.

\subsection{Description of main result}\label{descrip}

Our main result establishes the non-vanishing at points in the critical strip of Rankin-Selberg $L$-functions of $\GL_n$ cusp forms when they are twisted by a large, spectrally defined class of characters.  We will state this as Theorem \ref{non-van thm} in Section \ref{nonvanish} after the necessary notation has been developed.  Just to give the reader a preview of the contents of Theorem \ref{non-van thm}, we state below a special case, and in somewhat informal language, suppressing some technical conditions.

Let $\pi$ be a cusp form on $\GL_n$ over a number field $K$, where $n \geq 2$.  Fix an archimedean place $v$ of $K$.  Denote by $X$ the set of all Hecke characters whose conductor divides a fixed ideal $\cc$, whose component at $v$ is trivial, and whose components at archimedean places other than $v$ are restricted to a box around the origin.  Write $V$ for the total volume of all ramification conditions: this is roughly the norm of $\cc$ times the volume of the box.  Finally, fix any $\beta\in (1 - 2(n^2+1)^{-1},1)$.  Then, as long as the archimedean volume is taken large enough with respect to the field, we prove
\begin{equation*}
\sum_{\chi \in X} |L(\beta, \pi\otimes\chi\times\widetilde{\pi})| \gg_\varepsilon V^{1-\varepsilon}. 
\end{equation*}

Clearly, when $K=\Q$ or an imaginary quadratic field, the condition of being trivial at the one archimedean embedding renders meaningless the final condition on the other archimedean parameters.  In these cases, the set $X$ is the group of Dirichlet characters of conductor dividing $\cc$, and our result is not new (see \cite{LRS1} and the discussion after Theorem \ref{non-van thm}).  For all other fields, $X$ consists (primarily) of infinite order characters.  This greater generality  allows us, for instance, to take $\cc$ arbitrary in contrast to the non-vanishing result of \cite{LRS2}.  We present an application in  Section \ref{appl}.  The precise formulation of the above result, which we state in Theorem \ref{non-van thm}, seems to be the first non-vanishing theorem for twists by infinite order characters in the literature (although see \cite[Section 1.4]{MV}).

\subsection{Families of $L$-functions} In the branch of analytic number theory concerned with the analytic properties of $L$-functions associated to automorphic forms, the notion of a {\it family} of automorphic forms is a helpful  organizing principle.  Its aim is to group together ``like" forms according to sometimes well-defined, sometimes statistical or  phenomenological, shared traits.  Such traits can include the distribution of zeros of $L$-functions close to the point $s=1/2$ (the so called low lying zeros first studied by Iwaniec-Luo-Sarnak \cite{ILS}) and associated non-vanishing theorems at $s=1/2$, their functorial provenance such as base-change lifts, or simply the size of their conductor.  A recent preprint of Sarnak \cite{Sa} summarizes these various groupings, as well as many others.

The families of character twists occurring in our work arise as a spectral family.  They appear in the spectral support of idelic test functions defined by local ramification restrictions, once they are averaged over $K^\times$.  The process of defining these idelic test functions resembles the general Paley-Wiener theorem of Clozel-Delorme \cite{CD}, and we have borrowed various notations from this set-up.  In Section \ref{families} we develop a very general approach to defining spectrally generated families of characters.  In particular, the condition of being trivial at one given archimedean place, imposed in the informal discussion of Section \ref{descrip} is replaced by a more general linear condition.  For example, one could ask that the characters be trivial on the diagonally embedded copy of the positive reals, which is a rather standard normalization.  Indeed, in the case of $\cc=1$ this normalization recovers the characters of the extended class group $\widetilde{\rm Div}_K^0/K^\times$ of \cite{EV}.\footnote{We note that the definition Arakelov class group given by Schoof in \cite{Sch} is different and less robust than that of \cite{EV}, in that no use is made of the compact part of $K_v^\times$ for archimedean $v$.} 

What makes our family amenable to analysis and useful for applications is the equilibrium between the size of the family and their local ramification data, as well as their analytic conductor.  This is a quantification of the regularization effect on the ray class group mentioned in Section \ref{inf ord}, and it is the reason why no hypothesis need be imposed on the ideal $\cc$.  These two properties are recorded in Lemma \ref{lem1} and Remark \ref{cond rem}, respectively.

\subsection{Principal application}\label{appl} At this point we confine ourselves to explaining one of the consequences of Theorem \ref{non-van thm}.  An important observation of Luo, Rudnick, and Sarnak \cite{LRS1} shows that bounds towards the Ramanujan conjecture for the group $\GL_n$ can be deduced from the non-vanishing of certain $L$-functions.  Following their argument (which we shall briefly reprise in the paragraph preceding Remark \ref{cond rem}), we use Theorem \ref{non-van thm} to approach the Generalized Ramanujan Conjecture, thereby re-proving the central result of \cite{LRS2} at infinite places.  Our methods work equally well for ramified archimedean places, which was the case subsequently established by Müller-Speh in \cite{MS}.

For a place $v$ of $K$ let $m(\pi,v)=\max_j \sigma_\pi(v,j)$, if $\pi_v$ is non-tempered, and $m(\pi,v)=0$ otherwise.  For the notation relative to Langlands parameters, see Section \ref{local}.

\begin{cor}\label{rama} Let $\pi$ be a unitary cuspidal automorphic representation of $\GL_n(\A)$.  Let $v$ be an archimedean place. Then $m(\pi,v) \leq (1/2) - 1/(n^2+1)$.
\end{cor} 

The original proof of Corollary \ref{rama} in \cite{LRS2} uses Deligne's estimate \cite{De} for hyper-Kloosterman sums together with   Rohrlich's deep work \cite{R} which is based among other things on the large sieve and the Bombieri-Vinogradov theorem over number fields.  Our proof of Corollary \ref{rama} avoids all of these results.  In fact, the local estimates of Section \ref{local section} are elementary, being obtained by standard stationary phase type arguments for real and $p$-adic oscillatory integrals.\footnote{Although the degree of generality in which we state Theorem \ref{non-van thm} requires the use of profound results of Deligne, for the application to Corollary \ref{rama} his work is not needed.  We will discuss this aspect of our work more thoroughly in Remark \ref{c=1 rmk}.}  In this way the classical method of Landau \cite{La}, Rankin \cite{Ra}, and Serre \cite{Se} is put on the same footing as that of Luo-Rudnick-Sarnak \cite{LRS1, LRS2}: for the latter just like the former, no deep results in algebraic geometry are needed to obtain the same quality bounds.

We remark that one can easily modify our proof of Corollary \ref{rama} to apply to all places $v$, archimedean or not.  There are two ways to do this; we briefly describe them here although neither approach, for the sake of a simplified presentation, will be pursued in this paper.

The first would be to adopt the method in \cite{BB}, with the necessary modifications for the present context, replacing the symmetric square with the Rankin-Selberg $L$-function and eliminating the average over moduli.  This argument does not proceed by non-vanishing of $L$-functions but rather by averaging their Dirichlet series coefficients.  One may view this approach as a direct extension to all number fields and all places (without any loss in the degree of the number field) of Rankin's method \cite{Ra}, as generalized by Serre \cite{Se}.  See the beginning of Section \ref{local section} for more details.

Alternatively, one could deduce Corollary \ref{rama} for finite places from the non-vanishing of $L$-functions, as was done in \cite{LRS1} and \cite{LRS2}.  To do so, one would need to broaden the class of characters introduced in Section \ref{families} to include those that are trivial at a fixed \emph{finite} place.  We have avoided doing this here since such characters would not enjoy the equilibrium properties described in Lemma \ref{lem1} and Remark \ref{cond rem}, even after allowing for them to be of infinite order.

\subsection{Organization of paper} In Section \ref{notation} we review the necessary local and global structures associated to number fields.  In particular we define the class of Hecke characters used in our theorems and prove that their size behaves regularly with respect to the local ramification conditions.

In Section \ref{RS} we review the theory of Rankin-Selberg $L$-functions.

In Section \ref{nonvanish} we state our main theorem, Theorem \ref{non-van thm}, and comment on its relation to the existing literature.  The proof of Theorem \ref{non-van thm} derives from an application of the functional equation of Rankin-Selberg $L$-functions, along with the positivity of their coefficients, and local estimates on the resulting integral transforms.  Sections \ref{sumform}, \ref{local section}, and \ref{non-van section} then give the details of the proof of Theorem \ref{non-van thm}.

In Section \ref{sumform} we prove the summation formula which encodes the functional equation of the Rankin-Selberg $L$-functions $L(s,\pi\otimes\chi\times\widetilde{\pi})$ as $\chi$ ranges through all Hecke characters of prescribed ramification.

In Section \ref{local section} we prove the required estimates on the local transforms, and then package them together in the form of an $S$-adic statement.

Finally, in Section \ref{non-van section}, we put all the ingredients together to complete the proof.

\subsection*{Acknowledgements} Parts of this work were completed during visits to Université de Nancy, Universität Göttingen, the Institute for Advanced Study, and the Centre Interfacultaire Bernoulli of the EPFL.  We would like to thank these institutions for their financial support and the hospitable work environments they offered.

\section{Preliminaries}\label{notation}

Let $K$ be a number field of signature $(r_1,r_2)$.  Let $d=r_1+2r_2$ be the degree of $K$ over $\Q$ and $r=r_1+r_2$ the number of inequivalent archimedean embeddings.  Let $\mathcal{O}_K$ denote the ring of integers of $K$, $\mathcal{O}_K^\times$ the unit group, and $\mu$ the group of torsion units.  Let $Cl$ denote the class group of $K$ and put $h=|Cl|$.  The norm map is defined to be the completely multiplicative function on integral ideals whose value at prime ideals is the cardinality of the associated residue field.  For an integral ideal $\cc$ of $\mathcal{O}_K$ we write $\mathcal{N}(\cc)$ for the norm of $\cc$ and put $\phi(\cc)=| (\mathcal{O}_K/\cc)^\times|$. If $S$ is a finite set of places, we denote by $P_K(S)$ the group of fractional principal ideals prime to $S$. This is naturally isomorphic with $\mathcal{O}_S^{\times}\backslash K^{\times}$, where $\mathcal{O}_S^{\times}$ are the $S$-units of $K$. 
 
If $v$ is a place of $K$, we let $K_v$ be the completion of $K$ relative to the norm $|\cdot |_v$ induced by the normalized valuation at $v$.  If $v=\C$, then $|\cdot |_v=|\cdot |^2$, the square of the modulus. 

For $v\mid\infty$ let $U_v$ be $\{\pm 1\}$ or $U(1)$ accordingly to whether $v$ is real or complex, respectively.  Let $U_\infty=\prod_{v\mid\infty}U_v$ and $\a=\R^r$.  Then $K_\infty^{\times} = \prod_{v \mid \infty} K_v^{\times}$ decomposes as $U_{\infty} \times \R_{>0}^r$ and therefore one has an isomorphism $K_\infty^{\times} \simeq  U_\infty\times\a$.  We will write 
 \begin{equation}\label{log}
   \log:K_\infty^\times\rightarrow\a,  \quad (x_v)_{v \mid \infty} \mapsto (\log|x_v|_v)_{v \mid \infty}
 \end{equation}  
for the projection onto the second factor.

When $v=\p$ is finite we write $\mathcal{O}_\p$ for the ring of integers in $K_\p$ and $\wp$ for its unique maximal ideal.  Let $\varpi_\p$ be any generator of $\wp$.  Put $U_\p=\mathcal{O}_\p^\times$ for the unit group.  For an integer $e\geq 1$ the degree $e$ neighborhood $1+\wp^e$ of $1$ in $U_\p$ is denoted by $U_\p^{(e)}$.

Let $\A$ be the adele ring and $\I$ the idele group of $K$. Write $|\cdot |_\A=\prod_v |\cdot |_v$ for the idelic norm, and put $\I^1$ for the subgroup of norm $1$ ideles.  Let $\Delta:\R_{>0}\rightarrow\I$ be the norm preserving embedding sending the positive real number $t$ to the idele having all finite components equal to $1$ and the $v$ component, for $v\mid\infty$, equal to $t^{1/d}$.  

Denote the idele class group by $\mathscr{C}=K^\times\backslash\I$.  Since $K^\times\subset \I^1$, we may consider the quotient $\mathscr{C}^1=K^\times\backslash\I^1$, a compact group.  For $\cc$ an integral ideal of $\mathcal{O}_K$ we define $\mathscr{C}({\cc})=\mathscr{C}/U({\cc})$ and $\mathscr{C}^1({\cc})=\mathscr{C}^1/U({\cc})$, where
\begin{equation*}
U({\cc})=\prod_{\p\nmid\cc}U_\p\prod_{\p^{e_\p}||\cc}U_\p^{(e_\p)}.
\end{equation*}
These two groups, $\mathscr{C}({\cc})$ and $\mathscr{C}^1({\cc})$, may be given a more explicit description by means of the Strong Approximation Theorem for $\I$.  This states that
\begin{equation*}
\I=\bigsqcup_{i=1,\ldots ,h} K^\times a_i\prod_\p U_\p K_\infty^\times,
\end{equation*}
for elements $a_1,\ldots ,a_h\in\I_f$.  If we then put
\begin{equation*}
V({\cc})=\prod_{\p^{e_\p}||\cc}U_\p/U_\p^{(e_\p)}\times U_\infty,
\end{equation*}
we find that
\begin{equation}\label{iso0}
\mathscr{C}(\cc)\simeq\left(\mathcal{O}_K^\times\backslash (V(\cc)\times\a)\right)^h.
\end{equation}

\subsection{Characters} Let $\widehat{\I}$ (resp., $\widehat{\I}^1$) be the group of continuous unitary characters of $\I$ (resp., $\I^1$).  We may and will identify $\widehat{\I}^1$ (resp., $\widehat{\mathscr{C}^1}$) with the subgroup of $\widehat{\I}$ (resp., $\widehat{\mathscr{C}}$) consisting of characters trivial on $\Delta(\R_{>0})$.  When $\chi\in\widehat{\mathscr{C}}$ we agree to write $\chi=\omega |\cdot |_\A^{it}$ for the unique $t\in \R$ such that $\omega\in\widehat{\mathscr{C}^1}$.  We will see in a moment that the characters of $\mathscr{C}(\cc)$ are precisely the Hecke characters of conductor dividing $\cc$.

At archimedean $v$ a character $\chi_v\in\widehat{K_v^\times}$ takes the form
\begin{enumerate}
\item[$\cdot$] $\chi_v(x)=\text{sgn}(x)^{\delta_v} |x|^{i\tau_v}$ for some $\delta_v\in\{0,1\} $ and $\tau_v\in\R$ if $v=\R$,
\item[$\cdot$] $\chi_v(z)=(z/|z|)^{\delta_v}|z|^{2i\tau_v}$ for some $\delta_v\in\Z$ and $\tau_v\in\R$ if $v=\C$.
\end{enumerate}
In general we shall write
\begin{equation*}
\chi_v(x)=\delta_v(x)e^{i\tau_v(\log |x|_v)},
\end{equation*}
where the real number $\tau_v$ is now thought of as a linear map on the reals, and $\delta_v$, by slight abuse of notation, is in $\widehat{U}_v$.  A character of $U_v$ or $U_\infty$ will generally be denoted $\delta_v$ or $\delta_\infty$, respectively. We adopt the convention that a character of $V(\cc)$ will be denoted $\delta$, with no subscript.

Now let $\a^*={\rm Hom}(\a,\R)$.  A character $\chi$ of $K_{\infty}^{\times}$ is called unramified if it is trivial on $U_\infty$.  Such characters are obtained from $\a^*$ by setting $\chi(x)=e^{i\tau(\log x)}$ for $\tau\in\a^*$ and $x\in K_\infty^\times$, where we have made use of the definition \eqref{log}.

With this notation in place, we deduce from \eqref{iso0} that
\begin{equation}\label{iso1}
\widehat{\mathscr{C}(\cc)}\simeq \left(\{(\delta,\tau)\in \widehat{V(\cc)}\times  \a^*: e^{i\tau(\log u)}\delta(u)=1\; \forall\; u\in\mathcal{O}_K^\times\}\right)^h. 
\end{equation}
In classical language this corresponds to the well-known fact that a Hecke character of conductor dividing  $\cc$ is, up to a class group character, determined by a character of $(\mathcal{O}_K/\cc)^{\times}$, and a character of $K_{\infty}$, whose product is trivial on units. 
Fixing the principal branch of the logarithm  we may write the above condition as
\begin{equation}\label{compat}
\tau(\log u)\in -\arg \delta(u)+2\pi\Z\quad\text{for every }\; u\in \mathcal{O}_K^\times,
\end{equation}
where $\arg \delta(u)\in (-\pi,\pi)$.  For $\chi\in \widehat{\mathscr{C}(\cc)}$ we write $(\delta(\chi),\tau(\chi))\in\widehat{V(\cc)}\times\a^*$ for the coordinates of $\chi$.

We proceed to recall the definition of the analytic conductor of a character $\chi \in \widehat{\I}$.  We first recall the notion of the conductor of a local character $\chi_v$ of $K_v^\times$.  At finite places $v=\p$, one has
\begin{equation*}
\chi_\p(x)=\delta_\p(x) |x|_\p^{it_\p},
\end{equation*}
where $\delta_\p\in\widehat{U}_\p$ and $t_\p\in\R$ is well-defined up to a multiple of $2\pi/\log\mathcal{N}(\p)$.  The continuity of $\chi_\p$ implies the existence of a largest open compact subgroup $U_\p^{(r_\p)}$ of $U_\p$ on which $\delta_\p$ is trivial.  The local conductor at $v=\p$ is then
\begin{equation}\label{p cond}
\mathcal{C}(\chi_\p):=\mathcal{N}(\p)^{r_\p}.
\end{equation}

At archimedean $v$ the conductor of $\chi_v$ is taken to be
\begin{equation}\label{v cond}
\mathcal{C}(\chi_v):=(1+|\delta_v + i\tau_v|^{ [K_v : \Bbb{R}]}) \qquad \text{ for } v \mid\infty.
\end{equation}

Now let $\chi=\prod_v\chi_v \in\widehat{\mathscr{C}}$ be a (unitary) Hecke character.  For almost all $\p$ we have $\cc(\chi_{\p}) = \mathcal{O}_\p$ so that $\mathcal{C}(\chi_\p)=1$.  We write $\cc(\chi)=\prod\p^{r_\p}$, the product taken over all finite primes $\p$.  In the terminology of \cite{IS}, one says that $\chi$ has  analytic conductor $\mathcal{C}(\chi):=\prod_v \mathcal{C}(\chi_v)$.  It is expected that the analytic conductor is the proper measure of complexity of a Hecke character $\chi$ (and of automorphic forms in general).  

\subsection{Normalization of Haar measures}
We fix a non-trivial character $\psi_v$ of $K_v$ by taking $\psi_v(x)=\exp(2\pi i x)$ when $v=\R$, $\psi_v(x)=\exp(2\pi i(x+\overline{x}))$ when $v=\C$, and $\psi_\p$ the unique additive character trivial on $\mathfrak{d}_{\mathfrak{p}}^{-1}$ and on no larger subgroup when $v=\p$ is non-archimedean.  Let $dx_v$ be the self-dual Haar measure on $K_v$.  Explicitly, $dx_v$ is Lebesque measure if $v$ is real, twice the Lebesque measure if $v$ is complex, and the unique Haar measure such that $\mathcal{O}_\mathfrak{p}$ has volume $\mathcal{N}(\mathfrak{d}_\p)^{-1/2}$ if $v=\p$ is finite.  

On $K_v^{\times}$ we choose the normalized Haar measure $d^\times x_v = \zeta_v(1)dx_v/|x_v|_v$, where $\zeta_v$ is the Tate local zeta function at $v$. We let $d^\times x$ be the measure on $\I$ that on the standard basis of open sets of $\I$ coincides with $\prod_v d^\times x_v$; this descends to a quotient measure on $\mathscr{C}$.

The image measure of $d^\times x$ under the isomorphism
\begin{equation*}
\mathscr{C}\overset{\sim}{\longrightarrow}\mathscr{C}^1\times\R_{>0}
\end{equation*}
is the product of the Haar measure of volume $c_K^{-1}$ on the first factor, where $c_K^{-1} = \underset{s=1}{\rm Res}\; \zeta_K(s)$, and the multiplicative measure $t^{-1}dt$ on the second factor (see \cite[VII.6.\ Prop.\ 12]{We}).

Next we normalize measures on the relevant groups of characters.  We refer to \cite[Section 2]{BB} for the notation relative to Fourier-Mellin transforms.  We normalize the Haar measure $d\chi_v$ on $\widehat{K_v^\times}$ so as to recover the Mellin inversion formula.  Explicitly $d\chi_v$ is given by
\begin{displaymath}
c_v \sum_m \int_{(\sigma)} g(s, \eta_m) \frac{ds}{2\pi i}, \qquad \sum_{\eta \in \widehat{U_{\mathfrak{p}}}} \int_{\sigma -\frac{i\pi}{\log \mathcal{N}(\mathfrak{p})}}^{\sigma+\frac{i \pi}{\log \mathcal{N}(\mathfrak{p})}} g(s, \eta) \log\mathcal{N}(\mathfrak{p})\frac{ds}{2\pi i},
\end{displaymath}
for $v\mid\infty$ and $\mathfrak{p}$, respectively. Here $c_\R=1/2$ and $c_\C=1/(2\pi)$.  The corresponding measure on 
\begin{equation}\label{decomp}
  \widehat{\mathscr{C}}\simeq\widehat{\mathscr{C}^1}\times\R
\end{equation}  
is the product of $c_K$ times the counting measure on the first factor and $1/2\pi$ times Lebesgue measure on the second factor.

\subsection{Families of Hecke characters}\label{families}

The set of all unitary Hecke characters is a disjoint union of continuous families, each given by $\chi |\cdot|_\A^{it}$ for $t$ varying in $\R$.  The presence of these continuous families is a direct expression of the non-compactness of the idelic quotient $\mathscr{C}=K^\times\backslash\I$.  We index them by the discrete group $\widehat{\mathscr{C}^1}$.  Thus
\begin{equation*}
\widehat{\mathscr{C}}=\bigsqcup_{\omega\in\widehat{\mathscr{C}^1}}\{\omega |\cdot |_\A^{it}: t\in \R\}.
\end{equation*}

\subsubsection{Discrete subfamilies}\label{discrete} In this section we discuss a general method for isolating a discrete subset of $\widehat{\mathscr{C}}$ by means of linear constraints on the parameters $\tau_v$, for $v\mid\infty$.  This method is valid for all number fields but becomes trivial for fields with only one archimedean embedding.

We regard $\a$ and $\a^*$ as Euclidean spaces for the standard inner product.  Let $\a_0$ be the trace-zero hyperplane in $\a$ and let $\l_0$ be the orthogonal complement of $\a_0$ in $\a$.   Denote by $\a^*=\a_0^\perp\oplus \h_0$ the dual decomposition.  Thus $\h_0=\l_0^\perp$ is the trace-zero hyperplane in $\a^*$.  As $\tau$ runs over the line $\a_0^\perp=\R.(1,\ldots ,1)$, the characters $e^{i(\tau\circ\log)}$ describe the characters $|\cdot |_\infty^{it}$ for $t\in \R$, while as $\tau$ runs over $\h_0$, the characters $e^{i(\tau\circ\log)}$ comprise the totally unramified characters of $K_\infty^\times$ trivial on $\Delta(\R_+)$.

The Hecke characters having $\tau(\chi)\in\h_0$ are precisely those in $\widehat{\mathscr{C}^1}$.  As natural as this normalization is, however, it is important in applications to work with families whose exponents $\tau(\chi)$ are taken to lie in a more a general hyperplane $\h$.  On the other hand, if $\h$ contains the line of exponents $\a_0^\perp$, then the set of Hecke characters having exponents in $\h$ is not discrete.  We are lead therefore to consider the following definition.

A hyperplane $\h\subset \a^*$ is called {\it admissible} if $\a_0^\perp \subsetneq \h$.  

We fix now an admissible hyperplane $\h\subset\a^*$.  Write $\Lambda_0=\log\mathcal{O}_K^\times\subset\a_0$ and let
\begin{equation*}
\Lambda_0^\perp=\{\tau\in\a^*: \tau(x)\in 2\pi\Z\; \forall\; x\in\Lambda_0\}
\end{equation*}
be the dual of $\Lambda_0$ {\it within all of $\a^*$}.  It is a disjoint union of continuous families.  Let $\mathcal{L}_\h=\Lambda_0^\perp\cap\h$. If $\h=\l^\perp$ for a line $\l\subset\a$ then $\mathcal{L}_\h$ is the discrete family of $\tau\in\Lambda_0^\perp$ that are trivial on $\l$. See Figures 1 and 2 for an example in the real quadratic case.\\

\includegraphics[height=5.5cm]{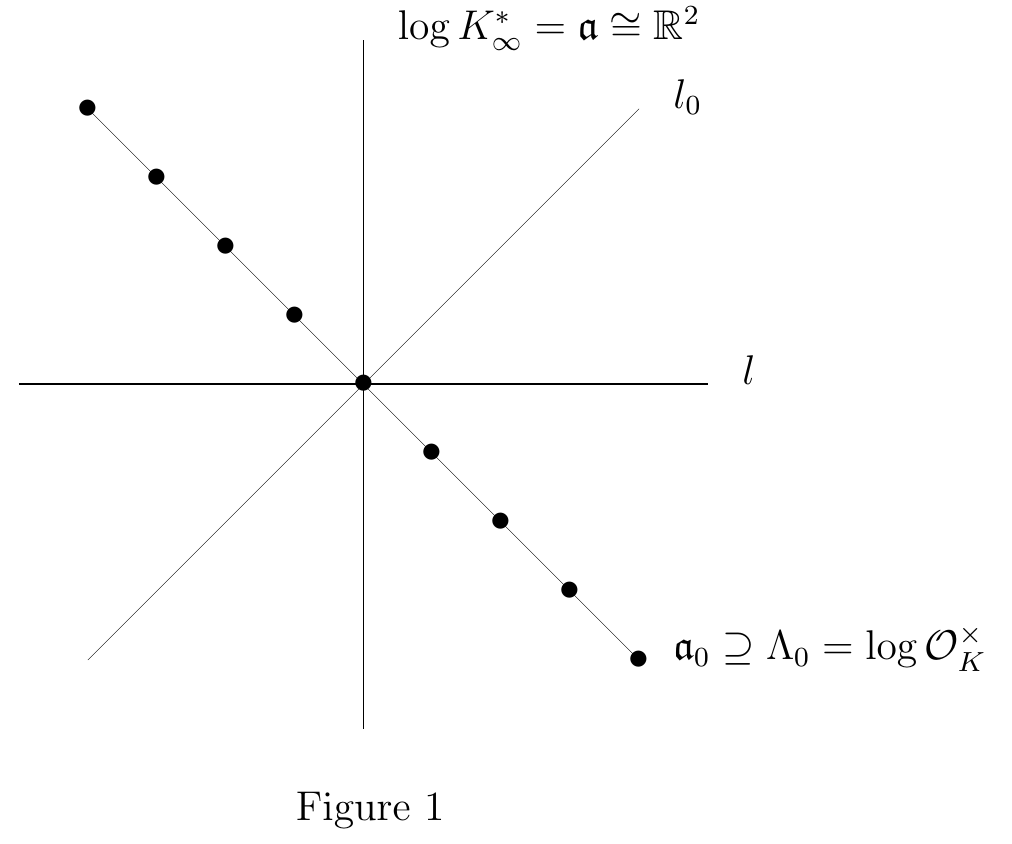}  \hspace{1cm} \includegraphics[height=5.5cm]{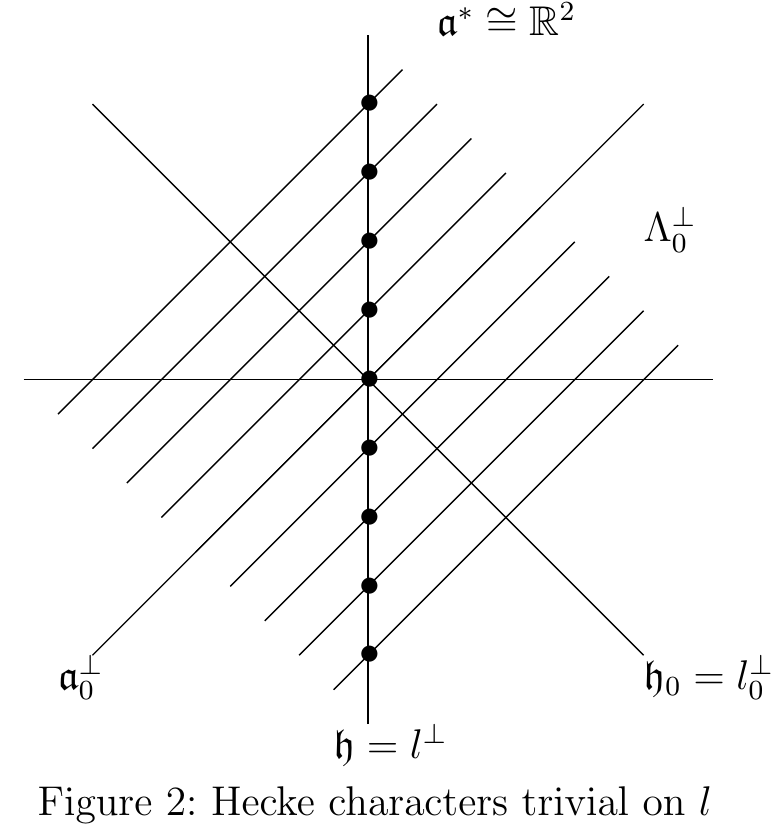}  \\

To be more concrete, let $\h$ be defined by the linear equation $\sum_v \alpha_v\tau_v=0$.  So $\h$ admissible means $\sum_v \alpha_v \not=0$.  Now let $u_1,\ldots ,u_{r-1}$ be a system of generators for $\mathcal{O}_K^\times/\mu$, and for $v\mid\infty$ denote by $u_j^{(v)}$ the image of $u_j$ in $K_v$.  Then the admissibility of $\h$ is equivalent to the regularity of the $r \times r$-matrix 
\begin{equation*}
M_\h:=\left(\; \begin{matrix} \boxed{(\log |u_j^{(v)}|_v)_{\substack{1 \leq j \leq r-1\\ v \mid\infty}}} \\ (\alpha_v)_{v \mid\infty} \end{matrix}\; \right).
\end{equation*}
(See \cite[Page 78]{Bu} for an example calculation for the totally real case.)  Finally the discreteness of $\mathcal{L}_\h$ is deduced from the invertibility of $M_\h$ since the volume of the fundamental mesh of $\mathcal{L}_\h$ is given by the absolute value of the determinant of $M_\h$.

\subsubsection{Prescribing ramification}  Fix a norm on $\a^*$ and for $T\geq 1$ let $B(0,T)\subset\a^*$ denote the ball of radius $T$ about the origin.  Let $\h\subset\a^*$ be an admissible hyperplane.  Let $\cc$ be an integral ideal of $\mathcal{O}_K$, and $D\subset \widehat{U}_\infty$ a finite subset.  We consider
\begin{equation}\label{def of X}
X(\cc,D,T)=\{\chi\in \widehat{\mathscr{C}(\cc)}:  \delta_\infty(\chi)\in D ,\; \tau(\chi)\in B(0,T)\cap\h\}.
\end{equation}
This is the most general form of families we will consider in this paper. Let
\begin{equation*}
V(\cc,D,T):=\phi(\cc)|D| {\rm vol}(B(0,T)\cap\h)
\end{equation*}
be the product of local volumes.  This quantity measures the size of the family of $\chi\in\widehat{\I}$ having ramification prescribed by $\cc$, $D$, and $B(0,T)\cap\h$ but subject to no global invariance requirement.  

\begin{lemma}\label{lem1} If $T \gg 1$ then $\left|X(\cc,D,T)\right| \asymp   V(\cc,D,T)$.\end{lemma}

In other words, Lemma \ref{lem1} states that for the families $X(\cc,D,T)$ the global obstruction to being trivial on $K^\times$ is analytically negligible.\footnote{In Lemma \ref{lem1}, and indeed throughout this entire paper, all implied constants are allowed to depend on the number field $K$ and the admissible hyperplane $\h$, which we consider as fixed.  We have suppressed this dependence on the notation for typographical simplicity.}

\begin{proof} For $\delta\in\widehat{V(\cc)}$ let $\Lambda_0^\perp(\delta)$ consist of all $\tau\in\a^*$ satisfying \eqref{compat}.  As in the discussion of \S\ref{discrete} we put $\mathcal{L}_\h(\delta)=\Lambda_0^\perp(\delta)\cap\h$; this is a shifted lattice in $\h$.  Recalling \eqref{iso1} we have
\begin{equation*}
|X(\cc,D,T)|=h |\{(\delta,\tau)\in\widehat{V(\cc)}\times\a^*: \delta_\infty\in D,\, \tau\in \mathcal{L}_\h(\delta)\cap B(0,T)\}|.
\end{equation*}
We estimate $|X(\cc,D,T)|$ by fibering over the set $D(\cc)=\{\delta \in \widehat{V(\cc)} \mid \delta_{\infty} \in  D\}$.  To this end, for $\delta\in \widehat{V(\cc)}$ we put $X(\delta,T)=\{\tau\in\a^*: \; \tau\in \mathcal{L}_\h(\delta)\cap B(0,T)\}$.   Thus
\begin{equation}\label{sum}
|X(\cc,\Omega,T)|=h\sum_{\delta\in D(\cc)} |X(\delta, T)|.
\end{equation}
An elementary lattice point argument (see e.g.\ \cite[Theorem 1.7]{Kr}) shows  that
\begin{equation*}
|X(\delta, T)| \sim c \cdot  {\rm vol}(B(0,T)\cap\h)\quad\text{as }\; T\rightarrow\infty,
\end{equation*}
for some constant $c > 0$ depending only on $K$ and $\h$.  In particular, since $\arg\delta(u)$ is confined to a bounded interval, if $T\gg 1$ is sufficiently large, then
\begin{equation*}
|X(\delta, T)| \asymp  {\rm vol}(B(0,T)\cap\h)
\end{equation*}
for all choices of $\delta$.  Inserting this into \eqref{sum} we find that for $T\gg 1$,
\begin{align*}
|X(\cc,\Omega,T)| &\asymp h|D(\cc)|{\rm vol}(B(0,T)\cap\h)\\
&=h\phi(\cc)|D|{\rm vol}(B(0,T)\cap\h)\asymp V(\cc,D,T),
\end{align*}
which proves the lemma.\end{proof}

\begin{rem} The ray class characters are precisely those characters $\chi=\prod_v\chi_v$ such that $\tau(\chi)=0$ and $\delta_v(\chi)=0$ for all complex $v$.  The group $X(\cc)$ of ray class characters of conductor dividing $\cc$ can be seen to fit into the above framework.  Indeed we can choose any admissible hyperplane $\h\subset\a^*$ and then the notation $X(\cc,\{0,1\}^{r_1},0)$ is nothing other than $X(\cc)$. One peculiarity of $X(\cc)$ is that it can dramatically fail to satisfy Lemma \ref{lem1}.  As discussed in Section \ref{inf ord}, it may well happen that even for large $\cc$ the set $X(\cc)$ consists only of the trivial character. This highlights the importance of the assumption $T\gg 1$ (recall that the implied constants depend on $\h$ and $K$). 

\end{rem}

\section{Rankin-Selberg $L$-functions}\label{RS}

Let $\pi\simeq\otimes_v\pi_v,\pi'\simeq\otimes\pi_v'$ be cuspidal automorphic representations of $\GL_n$ over $K$.  Most of the fundamental properties of the Rankin-Selberg $L$-function $L(s,\pi\times\pi')$ were established in the influential work \cite{JPSS}.  We briefly review some of these properties below.

\subsection{Local theory}\label{local}
For the definition of $L(s,\pi_v\times\pi_v')$ when $\pi_v$ and $\pi_v'$ are tempered, we refer the reader, for example, to the Appendix of \cite{RudSar}.  We now explain how to reduce the definition of the local Rankin-Selberg $L$-function for arbitrary pairs of generic unitary irreducible representations $\pi$ and $\pi'$ to that of tempered pairs.

For a general non-tempered representation, we have the following para\-metrization of $\pi_v$ as a Langlands quotient.  One associates with $\pi_v$ a standard parabolic subgroup $P$ of $\GL_n(K_v)$ of type $(n_1,\ldots , n_r)$ with unipotent radical $U$, irreducible tempered representations $\tau_j$ of $\GL_{n_j}(K_v)$, and real numbers $\sigma_j$ satisfying $1/2> \sigma_1>\cdots >\sigma_r> -1/2$.  These data satisfy the property that $\pi_v$ is equivalent to the fully induced representation 
\begin{equation*}
{\rm Ind}(\GL_n(K_v),P;{\boldsymbol\tau}[{\boldsymbol \sigma}]),
\end{equation*}
where ${\boldsymbol\tau}[{\boldsymbol \sigma}]$ is the representation of the group
\begin{equation*}
M=P/U\simeq \GL_{n_1}\times\cdots\times \GL_{n_r}
\end{equation*}
 given by
\begin{equation*}
{\boldsymbol\tau}[{\boldsymbol \sigma}]=\tau_1[\sigma_1]\otimes\cdots\otimes\tau_r[\sigma_r].
\end{equation*}
Here we have denoted by $\tau[\sigma]$ the twisted representation $g\mapsto \tau(g)|\det g|_v^\sigma$.  Note that by the unitarity of $\pi_v$ we have an equality $\{\tau_i[\sigma_i]\}=\{\widetilde{\tau_i}[-\sigma_i]\}$ as sets.  For the parametrization of $\pi'_v$ as a Langlands quotient, we simply add a prime $'$ to all objects in the above notation.

With the above conventions in place we can now write the local Rankin-Selberg $L$-function as a product
\begin{equation*}
L(s,\pi_v\times\pi'_v)=\prod_{i=1}^r\prod_{j=1}^{r'} L(s+\sigma_i+\sigma_j',\tau_i\times\tau'_j).
\end{equation*}
We single out the important property that $L(s,\tau_i\times\tau'_j)$ is holomorphic on $\Re(s)>0$.  From this (and the fact \cite{JS} that $0\leq |\sigma_i|,|\sigma'_j|< 1/2$) we deduce that $L(s,\pi_v\times\pi'_v)$ is holomorphic on $\Re(s)\geq 1$.  Let $m(\pi,v)=\max_j |\sigma_j|$ if $\pi$ is non-tempered, and $m(\pi,v)=0$ otherwise.  Then, in the case of contragredient pairs $\pi'_v\simeq\widetilde{\pi}_v$ the real number $2m(\pi,v)$ is precisely the right-most pole of $L(s,\pi_v\times\widetilde{\pi}_v)$.  This last observation will be used in the paragraph preceding Remark \ref{cond rem} below.

For finite places the local $L$-factor is of the form $P_\p(\mathcal{N}(\mathfrak{p})^{-s})$, where $P_\p$ is a polynomial of degree at most $n^2$ with $P_\p(0)=1$.  For infinite places there exist complex parameters $\mu_{\pi\times\pi'}(v,j)$, $j=1,\ldots ,n^2$, such that
\begin{equation*}
L (s,\pi_v\times\pi_v')=\prod_{j=1}^{n^2}\Gamma_v(s-\mu_{\pi\times\pi'}(v,j)),
\end{equation*}
where as usual we have written $\Gamma_\R(s)=\pi^{-s/2}\Gamma(s/2)$ and $\Gamma_\C(s)=2(2\pi)^{-s}\Gamma(s)$.  When $\pi'\simeq\widetilde{\pi}$ we have
\begin{equation}\label{mu-m}
2m(\pi,v)= \max_j \Re \mu_{\pi\times\widetilde{\pi}}(v,j).
\end{equation}

Recall that we have fixed once and for all the standard additive character $\psi_v$.  Having done so we may define the local $\gamma$-factor $\gamma (s,\pi_v\times\pi_v',\psi_v)$ which appears in the local functional equation
\begin{equation*}
L(s,\pi_v\times\pi'_v)=\gamma (s,\pi_v\times\pi_v',\psi_v)L(1-s,\widetilde{\pi}_v\times\widetilde{\pi}'_v).
\end{equation*}
One has
\begin{equation*}
\gamma (s,\pi_v\times\pi_v',\psi_v)=\epsilon (s,\pi_v\times\pi_v',\psi_v)L(1-s,\widetilde{\pi}_v\times\widetilde{\pi}'_v)/L(s,\pi_v\times\pi'_v)
\end{equation*}
for a canonically defined function $\epsilon (s,\pi_v\times\pi_v',\psi_v)$.

We now record several exact formulae for $\gamma (s,\pi_v\times\pi_v',\psi_v)$ when $\pi_v$ is twisted by a character.  Let $\p$ be a prime, unramified over $\Q$, and not dividing any of the finite places at which either $\pi$ or $\pi'$ ramifies.  Let $\chi_\p$ be a character of $K_\p^\times$ of conductor $\p^r$.  If $r=0$ then
\begin{equation}\label{unram chi gamma}
\gamma (s,\pi_\p\otimes\chi_\p\times\pi_\p',\psi_\p)=L(1-s,\widetilde{\pi}_\p\otimes\overline{\chi}_\p\times\widetilde{\pi}'_\p)/L(s,\pi_\p\otimes\chi_\p\times\pi'_\p),
\end{equation}
while if $r\geq 1$ we have
\begin{equation}\label{ram chi gamma}
\gamma (s,\pi_\p\otimes\chi_\p\times\pi_\p',\psi_\p)=\epsilon (s,\pi_\p\otimes\chi_\p\times\pi_\p',\psi_\p)=\mathcal{N}(\p^r)^{-n^2s}\mathcal{G}(\chi_\p)^{n^2},
\end{equation}
where $\mathcal{G}(\chi_\p)$ is the classical Gauss sum
\begin{equation*}
\mathcal{G}(\chi_\p)=\sum_{[u]\in U_\p/U_\p^{(r)}}\chi_\p(u)\overline{\psi_\p}(\varpi_\p^{-1} u).
\end{equation*}
For archimedean places $v\mid\infty$ and $\chi_v=\delta_v e^{i\tau_v}\in\widehat{K_v^\times}$ we have
\begin{align}\label{infinite gamma}
\gamma (s,\pi_v\otimes\chi_v\times\pi_v',\psi_v) =\epsilon (s,\pi_v\times & \pi_v',\psi_v)i^{|\delta_v|}\nonumber\\
&\times \prod_{j=1}^{n^2} \frac{\Gamma_v(1-s-i\tau_v - \mu_{\pi \times \widetilde{\pi}}(v, j) + |\delta_v|/\deg(v))}{\Gamma_v(s+i\tau_v - \bar{\mu}_{\pi \times \widetilde{\pi}}(v, j) + |\delta_v|/\deg(v))}.
\end{align}
This can be deduced, for example, from the explicit description of the local factors in \cite[Appendix]{RudSar}.

\subsection{Global theory} For $\Re(s)>1$ we define the completed Rankin-Selberg $L$-function to be the Euler product
\begin{equation*}
\Lambda(s,\pi\times\pi')=\prod_v L(s,\pi_v\times\pi'_v).
\end{equation*}
It admits a meromorphic continuation to a function of order 1 on the entire complex plane, and one has complete information on the location and order of its poles.  In fact, if $\pi'\not\simeq\widetilde{\pi}$ then $\Lambda(s,\pi\times\pi')$ is entire, and in the case of contragredient pairs, $\Lambda(s,\pi\times\widetilde{\pi})$ has simple poles at $s=1$ and $s=0$ and nowhere else. We call
\begin{equation*}
L(s,\pi\times\pi')=\prod_\p L(s,\pi_\p\times\pi'_\p)
\end{equation*}
the finite-part $L$-function, the product extending over all finite primes $\p$.  Since $\Lambda(s,\pi\times\pi')$ and the inverse of the Gamma function are functions of order 1, so too is $L(s,\pi\times\pi')$.

More generally, let $S$ be a finite set of places containing all infinite places as well as all places where $\pi$, $\pi'$, or the number field $K$ is ramified and put
\begin{equation*}
L^S(s,\pi\times\pi')=\prod_{\p\notin S}L(s,\pi_\p\times\pi_\p').
\end{equation*}
In the summation formula of Section \ref{sumform} we shall rely heavily on the (asymmetric) functional equation of $L^S(s,\pi\times\pi')$, which states that
\begin{equation}\label{fun'l eq}
L^S(s,\pi\times\pi')=\gamma_S(s,\pi\times\pi')L^S(1-s,\widetilde{\pi}\times\widetilde{\pi}'),
\end{equation}
where $\gamma_S(s,\pi\times\pi')=\prod_{v\in S}\gamma(s,\pi_v\times\pi_v',\psi_v)$.  Note that since $S$ contains all bad places the $\gamma_S$-factor is independent of the global character $\psi=\otimes_v\psi_v\in\widehat{\A/K}$.

Finally, let $\mathcal{C}(\pi\otimes\chi\times\widetilde{\pi})$ be the analytic conductor of the pair $(\pi\otimes\chi,\widetilde{\pi})$.  Computing explicit epsilon factors as in \cite{LRS1} one finds
\begin{equation}\label{RS cond}
\mathcal{C}(\pi\otimes\chi\times\widetilde{\pi})=\mathfrak{c}(\pi\times\widetilde{\pi})\mathfrak{c}(\chi)^{n^2}\mathcal{C}(\pi_\infty\otimes\chi_\infty\times\widetilde{\pi}_\infty)\ll_\pi \mathcal{C}(\chi)^{n^2}.
\end{equation}

\section{The Main Theorem}\label{nonvanish}

Let $||\tau||=\max_{v\mid\infty}\{|\tau_v|_v\}$ be the maximum norm on $\a^*$.  Thus for a parameter $T\geq 1$ we have
\begin{equation*}
B(0,T)=\{ \tau\in\h: |\tau_v|\leq T^{1/[K_v:\R]}\;\text{for every } v\}.
\end{equation*}
Let $\cc$ be any integral ideal of $\mathcal{O}_K$. Let $v_0 \mid\infty$ be an archimedean place.  Let $\h\subset\a^*$ be the admissible hyperplane defined by the linear condition $\tau_{v_0}=0$.  
Next put
\begin{equation*}
D=\{\delta_\infty\in\widehat{U}_\infty: \delta_{v_0}=0 \; \text{and}\; |\delta_v|\leq T^{1/2}\;\text{ for all complex } v\neq v_0\}.
\end{equation*}
Then using the definition \eqref{def of X} and the above input, the family $X(\cc,D,T)$ consists of all Hecke characters $\chi=\prod_v\chi_v$ of conductor dividing $\cc$, whose component $\chi_{v_0}$ is trivial, and whose archimedean parameters $\delta_v(\chi),\tau_v(\chi)$, for $v\neq v_0$, lie in the above boxes. In this case we have
\begin{equation}\label{V}
V(\cc,D,T)\asymp \phi(\cc)T^{r-1} .
\end{equation}
The following theorem is valid for the above choice of ramification data.

\bigskip

\begin{thm}\label{non-van thm} For $n \geq 2$ let $\pi$ be a unitary cuspidal automorphic representation of $\GL_n(\A)$.  Assume that $\cc$ is coprime to the finite places at which $\pi$ or $K$ is ramified and let $\nu$ denote the number of distinct prime ideal divisors of $\cc$. Let $\beta\in (1 - 2(n^2+1)^{-1},1)$.  Then for any $\varepsilon>0$ there exists $c=c(\varepsilon,K,\nu)>0$ such that for $T\geq c\mathcal{N}(\cc)^{\varepsilon}$ one has
\begin{equation}\label{sum over chi}
\sum_{\chi \in X(\cc,D, T)} |L(\beta, \pi\otimes\chi\times\widetilde{\pi})| \geq V(\cc,D,T)^{1-\varepsilon}. 
\end{equation}
\end{thm}

\bigskip

In the special case of $K=\Q$ the above theorem in principle reproduces the non-vanishing theorem of Luo-Rudnick-Sarnak in \cite[Proposition 3.1]{LRS1}.  Note however that the non-vanishing theorem in \cite{LRS1} involves an additional summation over moduli in a dyadic interval.  The advantage of this extra summation is to give an asymptotic formula for the average of $L$-values, but for the strict application to Ramanujan this is unnecessary.  It suffices to have a lower bound, and this can obtained by the positivity of coefficients of the Rankin-Selberg $L$-function.  See the recent book by Bergeron \cite[Proposition 7.39]{Ber} where this simplification is carried out for the special case when $n=2$.\\

The argument that deduces Corollary \ref{rama} from Theorem \ref{non-van thm} follows now verbatim the argument of \cite{LRS2} at the beginning of their Section 2. From \eqref{mu-m} we need to show that the local factor at $v_0$ of $\Lambda(s, \pi \times \tilde{\pi})$ has no pole on the segment  $1 - (n^2+1)^{-1}<s<1$. More generally, if $\chi$ is any Hecke character trivial at $v_0$, we need to show that the local factor at $v_0$ of $\Lambda(s,\pi\otimes\chi\times\widetilde{\pi})$ has no pole on this segment. Since the global  Rankin-Selberg $L$-function $\Lambda(s,\pi\otimes\chi\times\widetilde{\pi})$ is holomorphic (except possibly for a pole at $s=0$ or $s=1$) and the archimedean factors never vanish, it suffices to show for every $\beta\in (1-(n^2+1)^{-1},1)$ there exists $\chi$ trivial at $v_0$ such that $L(\beta,\pi\otimes\chi\times\widetilde{\pi})\neq 0$; this in turn is guaranteed by Theorem \ref{non-van thm}.

\begin{rem}\label{cond rem}
From the definitions \eqref{p cond} and \eqref{v cond} we see that
\begin{equation*}
\mathcal{C}(\chi) \ll \mathcal{N}(\cc)T^{r-1}
\end{equation*}
for every $\chi \in X(\cc,D, T)$.  By Lemma \ref{lem1} and \eqref{V} we have (with an implied constant depending on the number of prime factors of $\cc$)
\begin{equation*}
\left|X(\cc,D, T)\right|\asymp \mathcal{N}(\cc) T^{r-1}
\end{equation*}
so that $\mathcal{C}(\chi) \ll \left|X(\cc,D, T)\right|$ for every $\chi \in X(\cc,D, T)$.  We see then that the twisting family employed in Theorem \ref{non-van thm} not only satisfies the estimates in Lemma \ref{lem1}, but in addition the conductors of its members are majorized by its size.  Note that this is true \emph{only} for the choice of hyperplane $\h$ given by setting one of the coordinates equal to $0$.  Indeed, keeping the above maximum norm on $\a^*$, if $\h$ is any other choice of hyperplane (including the ``usual" convention $\h_0$ given by $\sum_v t_v = 0$), then the analytic conductor of the characters $\chi$ in this family satisfy
\begin{equation*}
\mathcal{C}(\chi)\ll \mathcal{N}(\cc)T^r,
\end{equation*}
which is a factor of $T^{(r-1)/r}$ times larger than $|X(\cc,D,T)|$.  Thus the hyperplanes given by $\tau_{v_0}=0$, which can be thought of as Weyl chamber walls inside $\a^*$, give rise to the most analytically well-behaved families.

The interest in having such well-behaved families for an arbitrary modulus $\cc$ is not academic.  Indeed the absence of such families was the principal obstacle to extending to arbitrary number fields the best known bounds -- due to Kim-Sarnak \cite{KimSa}-- towards the Ramanujan conjecture for the group $\GL_2$.  Over general number fields all previous approaches had relied upon Rohrlich's construction of special moduli \cite{R}.  This problem was successfully resolved in \cite{BB}, although we did so in the most direct way, without recourse to non-vanishing results.
\end{rem}

\begin{rem}\label{c=1 rmk}
When $\cc$ is taken to be square-full, that is, $\p\mid\mathfrak{c}\Rightarrow\p^2\mid \mathfrak{c}$, the local estimates of Section \ref{local section} can be done elementarily; otherwise, one must use Deligne's bounds \cite{De} for hyper-Kloosterman sums to deduce Theorem \ref{non-van thm}.  As the choice is ours, we can opt for the former and make the implication to Theorem \ref{rama} independent of Deligne's bounds.
\end{rem}

\begin{rem}\label{some Serre} It is of some historical interest to note that if $K$ has at least two archimedean places, we can choose $\mathfrak{c} = \textbf{1}$ in the statement of Theorem \ref{non-van thm} in which case our twisting family consists of characters unramified at finite places.  Then only the archimedean estimates of Section \ref{local section} are needed for Theorem \ref{non-van thm}, and the proofs of these use nothing more than standard bounds for oscillatory integrals.  Hence when applied to fields of rank $r\geq 2$ and characters of modulus $\mathfrak{c} = \textbf{1}$, our own deduction of Corollary \ref{rama} from Theorem \ref{non-van thm} gives the first proof of the $1/2-1/(n^2+1)$ bounds at an  archimedean place using only archimedean ramification. This route towards Ramanujan was taken by Serre \cite{Se}, but he could only deduce the bounds $1/2 - 1/(dn^2 + 1)$ for a number field of degree $d$.  See Section \ref{local section} for a description of how our method resolves this problem.
\end{rem}

From knowledge of upper bounds on each individual term in the sum \eqref{sum over chi}, we can obtain information on the number of non-vanishing members (all are of course non-zero under the Generalized Riemann Hypothesis).  To see this, note that any bound on the local parameters of (fixed) polynomial strength, such as the Jacquet-Shalika bounds \cite{JS}, can be inserted pointwise into the Dirichlet series to deduce $L(1+\varepsilon,\pi\otimes\chi\times\widetilde{\pi})\ll_{\pi,\varepsilon} 1$, uniformly in $\chi$.  Combining this majorization with the bounds \eqref{RS cond}, the functional equation of the Rankin-Selberg $L$-function \eqref{fun'l eq}, and the Phragm\`en-Lindelof convexity principle (recall $L(s,\pi\times\pi')$ is of order 1), we find
\begin{equation*}
L(\beta,\pi\otimes\chi\times\widetilde{\pi})\ll_{\pi,\varepsilon} \mathcal{C}(\chi)^{\frac{n^2}{2}(1-\beta)+\varepsilon}.
\end{equation*}
Thus, for $\beta>1-1/(n^2+1)$, we have
\begin{equation*}
L(\beta,\pi\otimes\chi\times\widetilde{\pi})\ll_{\pi,\varepsilon} T^{\frac{n^2}{2}(1-\beta)+\varepsilon}\leq T^{1-(n^2+1)^{-1}+\varepsilon}.
\end{equation*}
This leads to the following quantitative  result.

\begin{cor}\label{first cor} Keep the notation and assumptions of Theorem \ref{non-van thm}.  Then
\begin{equation*}
|\{\chi\in X(\cc,D, T):  L(\beta, \pi\otimes\chi\times\widetilde{\pi}) \not= 0\}|
\gg_\varepsilon V(\cc,D, T)^{\frac{1}{n^2+1}-\varepsilon}
\end{equation*}  
for every $\varepsilon>0$. 
\end{cor}

Corollary \ref{first cor} yields nowhere near a positive proportion of non-vanishing $L$-values.  On the other hand, it is valid for arbitrary number fields and certain $L$-functions of possibly very large degree.\\

It would be interesting to study this and related families further.  For instance, one could ask about the non-vanishing at the central point of the Hecke $L$-functions associated to $\chi\in X(\cc,D,T)$.  The results of \cite{Mi} should suffice to compute first and second moments over this family and to conclude  that $L(1/2, \chi) \not=0$ for $\gg V(\cc,D,T)^{1-\varepsilon}$ characters $\chi \in X(\cc,D,T)$. We leave this for future investigation. 
  
\section{A summation formula}\label{sumform}

The goal of this section is to establish a summation formula that lies at the heart of our proof of Theorem \ref{non-van thm}.  One may give various names to this formula, such as Voronoi summation or approximate functional equation for Rankin-Selberg $L$-functions.  In any case, the formula takes as input factorizable functions $g=\otimes_v g_v$ on the ideles $\I$.  The summation formula will then output a relation between two smoothened sums of Dirichlet series coefficients of the Rankin-Selberg $L$-function and a smooth average of this same $L$-function twisted by Hecke characters, evaluated at the point $s=\beta$.

Let $B_{\pi,K}$ the set of finite places where $\pi$ or $K$ is ramified. In what follows we fix an integral ideal $\cc$ of $\mathcal{O}_K$, written $\cc=\prod \p^{e_\p}$, assumed to be prime to all places in $B_{\pi,K}$.   We  put $S=\{\p\mid\cc\}\cup B_{\pi,K} \cup \infty$ and fix  $v_0\mid\infty$.

\subsection{Defining $g$}\label{g}
We now define functions $g_v$ at each place. More general assumptions are ceratinly possible, but the following definition suffices for our application. 

\subsubsection{Test functions}
At all primes places $v \in S \setminus \{v_0\}$ we choose a test function $g_v\in C_c^\infty(K_v^\times)$.  Hence for archimedean $v\not= v_0$, the transform $\widehat{g}_v(\sigma,\chi_v)=\widehat{g}_v(\sigma,e^{i\tau_v}\delta_v)=\widehat{g}_v(s, \delta_v)$ is entire in the complex variable $s=\sigma+i\tau_v$ and decays rapidly in vertical strips.  Moreover, if $v = \Bbb{C}$, then $\widehat{g}_v(s, \delta_v)$ is also rapidly decaying in $\delta_v \in \Bbb{Z}$.   At finite primes in $S$ we impose a ramification condition:  we require $g_\p$ to be invariant under $U_\p^{(e_\p)}$ for $\p\mid\cc$ and under $U_\p^{(1)}$ for $\p \in B_{\pi,K}$.
 
\subsubsection{Function at $v=v_0$}\label{v zero}

At $v_0$ we let $g_{v_0}(x) = |x|_{v_0}^{-\beta} g_0(|x|_{v_0})$, where $g_0\in C_c^\infty(\R_{\geq 0})$ is a fixed non-negative smooth function with support in $[0, 1]$, satisfying $g_0(0) = 1$, and whose (right-) derivatives at $0$ vanish to all orders.   

Since $g_{v_0}$ is invariant under $U_v$, the transform $\widehat{g}_{v_0}(s,\delta_{v_0})$ is nonzero only for $\delta_{v_0}=0$ corresponding to the trivial character.  Moreover $\widehat{g}_{v_0}(s,0)$ is holomorphic in $s$ except for a simple pole at $s=\beta$ and decays rapidly in vertical strips away from the pole $\beta$.

\subsubsection{Coefficient function}\label{coeff-function}

Finally, for all $\p\notin S$ we define $g_\p$ in such a way as to recover the Dirichlet series coefficients $\lambda_{\pi\times\widetilde\pi}(\p^r)$ when these latter are thought of as $U_\p$-invariant functions on $K_\p^\times$.  To this end, for every $\p\notin S$ we put 
\begin{equation}\label{coeff fun}
g_\p(x)=\begin{cases} \lambda_{\pi\times\widetilde\pi}(\p^r),& v_\p(x)=r\geq 0,\\
					     0, & v_\p(x)<0.
			\end{cases}
\end{equation}
We colloquially refer to this choice of $g_\p$ as the coefficient function. An easy calculation shows that
\begin{equation*}
\widehat{g}_\p(s,\delta_\p)=\begin{cases} L(s,\pi_\p\times\widetilde{\pi}_\p),& \text{if }\, \delta_\p=1,\\
							   0, & \text{else}.
				     \end{cases}
\end{equation*}

\medskip

For all but finitely many primes $\p$ (namely, for all $\p\notin S$) $g_\p$ is clearly invariant under $U_\p$.  It then makes sense to consider
\begin{equation*}
g_S= \prod_{v \in S} g_v\quad\text{and}\quad g= \underset{\p\notin S}{\otimes} g_\p \times g_S.
\end{equation*}

\subsection{Defining $g^*$}

Similarly to the above we shall now define a function $g_v^*$ at each place, a sort of transform of $g_v$ defined via the $\gamma$-factor appearing in the local functional equation of $L(s,\pi_v\times\widetilde{\pi}_v)$.

For a place $v$ of $K$ and  $g_v\in C^\infty(K_v^\times)$ as defined above we put
\begin{equation}\label{defghatstar}
g_v^*(x)=\int_{\widehat{K_v^\times}} \widehat{g}_v(1-\sigma,\chi_v^{-1}) \gamma(1-\sigma, \pi_v\otimes\chi_v\times\widetilde{\pi}_v,\psi_v) \chi_v^{-1}(x)|x|_v^{-\sigma} d\chi_v
\end{equation}
where $\sigma>1$.  The integral is absolutely convergent. By Mellin inversion
\begin{equation*}
\widehat{g_v^*}(\sigma,\chi_v)=\widehat{g_v}(1-\sigma,\chi_v^{-1}) \gamma(1-\sigma, \pi_v\otimes\chi_v\times\widetilde{\pi}_v,\psi_v).
\end{equation*}
If $\p\notin S$ and $g_\p$ is defined as in \eqref{coeff fun}, then  $g_\p^*$ is $U_\p$-invariant (in fact, $g^*_\p = g_\p$).  From this we deduce that with the choice of $g=\otimes_v g_v$ as in \S\ref{g}, it makes sense to consider
\begin{equation*}
g_S^*= \prod_{v \in S} g_v^*\quad\text{and}\quad g^*= \underset{\p\notin S}{\otimes}g_\p^*\times g_S^*.
\end{equation*}

\subsection{Summation formula}

We let
\begin{equation}\label{capital Gs}
G_S(x) = \sum_{u \in \mathcal{O}_S^{\times}} g_S(ux)\quad\text{and}\quad G(x)= \sum_{\gamma \in K^{\times}} g(\gamma x),
\end{equation}
obtaining well-defined functions on $\mathcal{O}_S^{\times}\backslash K_S^\times$ and $\mathscr{C}$, respectively.  In the same way one defines $G_S^*(x)$ and $G^*(x)$ as sums over $S$-units and non-zero field elements. Almost verbatim as in \cite[Lemma 4, (2.3)]{BB} one sees that the the sums are absolutely convergent and that the Mellin inversion formula holds for $G$ and $G^*$; in fact we have 
\begin{displaymath}
G(x) = \int_{\widehat{\mathscr{C}}} \widehat{g}(\sigma, \chi)\chi^{-1}(x)|x|_\A^{-\sigma} d\chi,\qquad  G^*(x)=\int_{\widehat{\mathscr{C}}}\widehat{g^*}(\sigma, \chi)  \chi^{-1}(x)|x|_\A^{-\sigma} d\chi
 \end{displaymath}
for any $\sigma>1$.  Moreover, when $x_v=1$ for all $v\notin S$, the sums $G(x)$ and $G^*(x)$ may be written as smooth sums of Dirichlet coefficients.  Indeed, it is not hard to see that
\begin{equation}\label{nothard1}
  G(x) = \sum_{ \mathfrak{a} = (\alpha) \in P_K(S)} \lambda_{\pi \times \widetilde{\pi}}(\mathfrak{a}) G_S(\alpha x_S)
\end{equation}
and
\begin{equation}\label{nothard2}
G^{\ast}(1/x) = \sum_{\mathfrak{a} = (\alpha) \in P_K(S)} \lambda_{\pi \times \tilde{\pi}}(\mathfrak{a}) G_S^{\ast}(\alpha x_S^{-1}).
\end{equation}

\begin{prop}\label{Voronoi} We have
\begin{equation*}
G(x)=|x|_\A^{-1}R+ |x|_\A^{-1} G^{\ast}(1/x),
\end{equation*}
where
\begin{equation*}
R=c_K\sum_{\omega \in \widehat{\mathscr{C}^1}(\mathfrak{c})} \bar{\omega}(x)\sum_\rho \underset{s=\rho}{\rm Res}\; L^S(s, \pi\otimes\omega\times\widetilde{\pi}) |x|_{\Bbb{A}}^{1-s} \widehat{g}_S(s,\omega),
\end{equation*}
the sum over $\rho$ running over all poles of the integrand in $\Re (s)\in [0,1]$. Possible poles can occur only at $s=0$, $s=1$ (if $\omega$ is trivial), or $\Re s = \beta$, and the infinite sum over $\rho$ is absolutely convergent. 
\end{prop}

\begin{proof} 
By Mellin inversion and \eqref{decomp}  we have
\begin{displaymath}
\begin{split}
  G(x) & =  c_K\sum_{\omega \in \widehat{\mathscr{C}^1}}\bar{\omega}(x)\int_{(2)}\widehat{g}(s, \omega)|x|_{\Bbb{A}}^{-s} \frac{ds}{2\pi i}\\
  & =   c_K\sum_{\omega \in \widehat{\mathscr{C}^1(\cc)}}\bar{\omega}(x)\int_{(2)}L^S(s, \pi\otimes\omega\times\widetilde{\pi}) |x|_{\Bbb{A}}^{-s} \widehat{g}_S(s, \omega) \frac{ds}{2\pi i}.
  \end{split} 
\end{displaymath}
For each $\omega$ we shift the contour to $\Re s  = -1$. This is admissible by the rapid decay of the infinite components of $\widehat{g}(s, \omega)$ along vertical lines.  In doing so, pick up the the residue of the integrand coming from the pole at $s=1$ (and possibly at $s=0$) of $L^S(s, \pi\times\widetilde{\pi})$ as well as those coming from the function $\widehat{g}_{v_0}(s)$ at $\Re s = \beta$. We obtain
\begin{equation}\label{shifty}
G(x)=|x|_\A^{-1}R+ c_K\sum_{\omega \in \widehat{\mathscr{C}^1(\cc)}}\bar{\omega}(x)\int_{(-1)}L^S(s, \pi\otimes\omega\times\widetilde{\pi}) |x|_{\Bbb{A}}^{-s} \widehat{g}_S(s, \omega) \frac{ds}{2\pi i}.
\end{equation}
By the rapid decay of $\widehat{g}_S(s, \omega)$, the infinite sum $R$ converges absolutely. 
We apply the functional equation and change variables $(s, \omega) \mapsto (1-s, \bar{\omega})$. An application of the inverse Mellin transform shows that the remaining integral is precisely $|x|_\A^{-1}G^*(1/x)$.
\end{proof}

\section{Local and $S$-adic estimates}\label{local section}

This is a technical section that establishes certain estimates that we shall need in Section \ref{non-van section} when we prove Theorem \ref{non-van thm}.  Proposition \ref{local lemma} estimates (and in certain cases evaluates) the transforms $g_v^*$ for explicit choices of $g_v$ to be prescribed below.  These transforms are oscillatory integrals, and we bring several tools to bear to examine their size.  One of these is stationary phase, which in a sense replaces Landau's lemma \cite{La}.  Another is Deligne's theorem, which we need only when the modulus $\cc$ is not square-full.  We then go on to average these estimates over $S$-units, obtaining Corollary \ref{G star}, with which we end the section.

The archimedean estimates in this section describe how the loss in the degree of the number field observed by Serre (cf. Remark \ref{some Serre}) can be repaired.  An illustrative example is that of an imaginary quadratic field, which is treated in the $v=\C$ computation below.  Here, rather than using a a test function $g_v$ supported on annuli, we take a function supported in a small ball around 1 and then dilate it, thereby approximating an annular sector.  The resulting transform $g_v^*$ then involves an additional summation over the character group of the circle.  We convert this $\Z$-sum by Posson summation into a dual $\Z$-sum which is essentially supported on the first term.  As a result, a stationary phase argument in 1 dimension is replaced by one of 2 dimensions, gaining back the loss by a factor of 2.  This observation extends to all number fields by factorization of test functions.

\subsection{Local estimates}

As in the beginning of Section \ref{sumform}, let $\cc$ be an integral ideal of $\mathcal{O}_K$, written $\cc = \prod \p^{e_{\p}}$, and put $S=\{\p\mid \cc\}\cup B_{\pi,K} \cup\infty$. Let $T \geq 1$ be a parameter.   For $v \in S \setminus \{v_0\}$ we define $g_v$ as follows: 
\begin{itemize}
\item for $\p \mid\cc$ let $g_{\p}$ be the characteristic function on $U_\p^{(e_\p)}$;
\item for $\p \in B_{\pi,K}$ let $g_{\p}$ be the characteristic function on $U_{\p}^{(1)}$;
\item for archimedean $v \not=v_0$ let $g_v(x) = g_0(T|x-1|_v)$.  
\end{itemize} 
For  $v \not\in S \setminus \{v_0\}$ we define $g_v$ as in subsections \ref{v zero} and \ref{coeff-function}. 
The function $g_v$ is $U_v$-invariant for every $v$. Except at the place $v_0$ the function $g_v$ is of compact support. 

We begin by evaluating or estimating the Fourier transforms $\widehat{g_v}$ of the above defined functions.  It is easy to see that for $\p\mid\cc$
\begin{equation}\label{hatgfin}
  \widehat{g}_{\p} = \phi(\p^{e_{\p}})^{-1} \textbf{1}_{\text{deg}(\chi) \leq e_{\p}}.
\end{equation}
Moreover, $\widehat{g}_{v_0}(s,\delta)$ is nonzero only for $\delta=0$, holomorphic in $s$ except for a simple pole  at $s=\beta$, and satisfies the bound
\begin{equation}\label{v0decay}
  \widehat{g}_{v_0}(\sigma + i\tau, 0) \ll_{\sigma, A} (1+|\tau|)^{-A}, \quad |s-\beta| \geq 1/100
\end{equation}
for any $A \geq 0$. For archimedean $v\not= v_0$, the transform $\widehat{g_v}(s, \delta)$ is entire in $s$ and satisfies the bound
\begin{equation}\label{g*inf}
  \widehat{g_v}(\sigma + i\tau, \delta) \ll_{\sigma, A} \frac{1}{T}\left(1 + \frac{|\delta+i\tau|^{[K_v:\R]}}{T}\right)^{-A}
\end{equation}
for any $A \geq 0$. We now estimate the transforms $g^*_v$.

\begin{prop}\label{local lemma}  For any $\varepsilon > 0$, $A \geq 0$ the following bounds hold:
\begin{equation}\label{finite}
  g_{\p}^{\ast}(x) \ll  \frac{1}{\phi(\p^{e_\p})|x|_\p } \cdot \begin{cases} (1+  |x|_\p^{\frac{1}{2} + \frac{1}{2n^2} }), & |x|_\p \leq \mathcal{N}(\p)^{e_\p n^2},\\
  0, & |x|_\p > \mathcal{N}(\p)^{e_\p n^2},
\end{cases}
\end{equation}
for $\p\mid \cc$;
\begin{equation}\label{infinite}
  g_v^{\ast}(x) \ll_{A, \varepsilon}  \frac{1}{T|x|_v }\left(1 + |x|_v^{\frac{1}{2} + \frac{1}{2n^2}+\varepsilon}\right)    \left(1+ \frac{|x|_v}{T^{n^2+\varepsilon}}\right)^{-A}
  \end{equation}
for archimedean $v\not= v_0$;  and 
\begin{equation}\label{v0}
  g_{v}^{\ast}(x) \ll_{A,   v}  |x|_{v}^{-1  } (1+ |x|_{v})^{-A}
  \end{equation}
 if $v=v_0$ and if $v=\mathfrak{p} \in B_{\pi,K}$  is a finite ramified prime. 
\end{prop}

\subsubsection{Trivial estimate} We start with the proof of \eqref{v0}. By the Jacquet-Shalika bounds $\gamma(1-s, \pi_v \otimes \chi_v \times \pi_v', \psi_v)$ is holomorphic in $\Re s \geq 1$, and for $v \mid \infty$  we have by  \eqref{infinite gamma} and a crude form of Stirling's formula  the bounds
\begin{equation}\label{crudestirling}
 \gamma(1-s,\pi_v\otimes\delta_v\times\widetilde{\pi}_v,\psi_v) \ll_{\pi_v} \begin{cases} (1 +|\tau|)^{n^2(\sigma - \frac{1}{2})}, & v=\R;\\
 																(\tau^2 + \delta_v^2)^{n^2(\sigma-\frac{1}{2})}, & v=\C, 
													\end{cases}
\end{equation} 
where $s = \sigma + i\tau$. 

If $v = v_0$, then by \eqref{v0decay} we can shift the contour in \eqref{defghatstar} to $\Re s = 1$ or $1  + A$, establishing \eqref{v0} in the case $v = v_0$. If $v$ is a finite ramified place, then by construction $\widehat{g}_{\p}$ is supported on characters of degree at most 1, and the same contour shift followed by a trivial estimate establishes the required bounds. 

It remains to prove \eqref{finite} and \eqref{infinite} which we do in the next to subsections. 

\subsubsection{Non-archimedean case} Using the explicit formula for $\gamma(1-s, \pi_\p\otimes \chi_\p \times\widetilde{\pi}_\p,\psi_\p)$ given in \eqref{unram chi gamma} and \eqref{ram chi gamma} we write
\begin{displaymath}
  g_{\mathfrak{p}}^{\ast}(x) =  \frac{1}{\phi(\mathfrak{p}^{e_{\mathfrak{p}}})} \sum_{0\leq \nu \leq e_{\mathfrak{p}}} A_{\nu}(x)
\end{displaymath}
where
\begin{displaymath}
  A_0(x) = \int_{\sigma - i\pi/\log \mathcal{N}(\mathfrak{p})}^{\sigma + i\pi/\log \mathcal{N}(\mathfrak{p})} \frac{L(s, \widetilde{\pi}_\p \otimes \overline{\chi}_\p \times \pi_\p)}{L(1-s, \pi_{\mathfrak{p}} \otimes \chi_\p \times\widetilde{\pi}_\p)} |x|_{\mathfrak{p}}^{-s} \log \mathcal{N}(\p) \frac{ds}{2\pi i}
\end{displaymath}
and 
\begin{displaymath}
  A_{\nu}(x) = \sum_{\deg(\delta_{\mathfrak{p}}) = \nu} \mathcal{G}(\delta_{\p}) \bar{\delta}_{\p}(x/|x|_{\p})  \int_{\sigma - i\pi/\log \mathcal{N}(\mathfrak{p})}^{\sigma + i\pi/\log \mathcal{N}(\mathfrak{p})} \mathcal{N}(\p^{e_{\p}})^{n^2(s-1)} |x|_{\mathfrak{p}}^{-s} \log \mathcal{N}(\p) \frac{ds}{2\pi i}
\end{displaymath}
for $\nu > 0$. We recall that $L(s, \widetilde{\pi}_\p \otimes \overline{\chi}_\p \times \pi_\p)/L(1-s, \pi_\p\otimes\chi_\p\times\widetilde{\pi}_\p)    = P(\mathcal{N}(\p)^s)/Q(\mathcal{N}(\p)^{-s})$ where $P, Q$ are two polynomials of degree $n^2$ and $Q(0) = 1$. Hence for $\Re s$ sufficiently large we can expand $1/Q$ into an absolutely convergent power series in $\mathcal{N}(\p)^{-s}$, integrate term by term, and a standard application of the residue theorem shows that $A_0(x) = 0$ if $|x|_{\p} > \mathcal{N}(\p)^{n^2}$. For $1 < |x|_{\p} \leq \mathcal{N}(\p)^{n^2}$ we shift the line of integration to some very large $A$, getting a negligible contribution. For $|x|_{\p} \leq 1$ we shift to $\sigma = 1$ and estimate trivially  $A_0(x) \ll   |x|_{\p}^{-1}$, $|x|_{\p} \leq 1$. 

We proceed to bound the terms $A_{\nu}(x)$ for $\nu \geq 1$ and distinguish two cases. If $e_{\p} = 1$, the integral vanishes unless $|x|_{\p} = \mathcal{N}(\p)^{n^2}$ in which case we get, by Deligne's bounds \cite{De} for Hyper-Kloosterman sums,
\begin{displaymath}
\begin{split}
   |x|_{\p} A_{1}(x) & =  \sum_{\deg(\delta_{\mathfrak{p}}) = 1} \mathcal{G}(\delta_{\p})^{n^2} \bar{\delta}_{\p}(x/|x|_{\p})    =    \sum_{\deg(\delta_{\mathfrak{p}}) \leq 1} \mathcal{G}(\delta_{\p})^{n^2} \bar{\delta}_{\p}(x/|x|_{\p}) -1 \\
   & = \phi(\p)\sum_{\substack{x_1, \ldots ,x_{n^2} \in U_{\p}\\ x_1 \cdots x_{n^2} = x/|x|_{\p}}} \psi_{\p}(\varpi_{\p}^{-1}(x_1 +\ldots + x_{n^2})) - 1 \ll \mathcal{N}(\p)^{(n^2+1)/2}=|x|_\p^{\frac{1}{2}+\frac{1}{2n^2}},
\end{split}    
\end{displaymath}
proving \eqref{finite}.

If $e_{\p} \geq 2$, \eqref{finite} can be established in an elementary manner. Again the integral vanishes unless $|x|_{\p} = \mathcal{N}(\p)^{e_\p n^2}$. It is now easy to see that the terms $A_{\nu}(x)$, $1 \leq \nu < e_{\p}$, vanish, and the term $A_{e_{\p}}$ can be bounded by a routine calculation. For $K=\Bbb{Q}$ complete details can be found in \cite{BB1}, and the general case differs only by notational changes.

\subsubsection{Archimedean case}
We begin by recalling the definition of $g^*_v$.  We have
\begin{equation}\label{g star sum}
g_{v}^{\ast}(x) = c_v \sum_{\delta_v\in \widehat{U_v}}\delta_v(x)^{-1} I(x,\delta_v),
\end{equation}
where $c_v=1/2$ or $1/(2\pi)$ for $v$ real or complex, respectively, and for $\sigma$ large enough,
\begin{equation}\label{ga}
I(x,\delta_v)=\int_{(\sigma)} \widehat{g}_v(1-s, \delta_v) \gamma(1-s,\pi_v\otimes \delta_v\times\widetilde{\pi}_v,\psi_v) |x|_v^{-s} \frac{ds}{2\pi i}.
\end{equation}
In fact, one can take $\sigma\geq 1$ by the Jacquet-Shalika bounds.  As usual we identify $\delta_v$ with the corresponding integer in $\{0,1\}$ or $\Z$ according to whether $v$ is real or complex, respectively.

When $|x|_v\leq 1$ we shift the contour in $I(x,\delta_v)$ to $\sigma = 1$ and obtain by a trivial estimate $g_v^{\ast}(x) \ll  |x|_v^{-1} T^{-1}$.  To handle large $|x|_v$, we first let $\varepsilon>0$ and $A>0$ be as in the statement of Proposition \ref{local lemma}.  By \eqref{crudestirling} and \eqref{g*inf} we obtain $g_{v}^{\ast}(x) \ll |x|_v^{-A} T^{n^2(\sigma - \frac{1}{2})}$.  If $\sigma$ is sufficiently large with respect to $\varepsilon$ and $A$ then this proves \eqref{infinite} in the range $|x|_v\geq T^{n^2 +\varepsilon}$.

It therefore remains to prove
\begin{equation}\label{intermediate}
g_v^{\ast}(x) \ll_\varepsilon  T^{-1}|x|_v^{-\frac{1}{2} + \frac{1}{2n^2}+\varepsilon},\quad 1<|x|_v<T^{n^2+\varepsilon},
\end{equation}
when $v\neq v_0$.  For convenience we fix $\sigma=1$ in the integral $I(x,\delta_v)$.  The proof of \eqref{intermediate} relies on the principle of stationary phase.  See, for example, \cite[Proposition 5.2]{DS} for a proof of the following result.

\begin{lemma}\label{statphase} Let $u\in C_c^{\infty}(\R^d)$ have support in a compact set $K$.  Let $\phi$ be a smooth real-valued function on $\R^d$ having a single non-degenerate critical point $y_0 \in K$ and no other critical point.  There is a constant $C$, depending only on $d$ and $K$, such that for $\lambda \geq 1$,
\begin{displaymath}
  \bigg|\int_K e^{i \lambda \phi(y)} u(y) dy - A\lambda^{-d/2} \bigg| \leq C\lambda^{-1-d/2} \| u\|_{S^{3+d}}.
\end{displaymath}
Here $A=(2\pi)^{d/2}u(y_0)\exp(i\frac{\pi}{4}{\rm sgn}(\phi(y_0)) + i\lambda\phi(y_0))\, |\det Q_\phi(y_0)|^{-1/2}$, with $Q_\phi$ the Hessian of $\phi$, and $\|u\|_{S^{k}}$ the $k$-th Sobolev norm of $u$.
\end{lemma}

Putting the integrals defining $g^*_v$ into the form required by Lemma \ref{statphase} proceeds essentially in two steps.  The first is to use Stirling's formula to explicate the corresponding phase function $\phi$; the second is to reduce the domain of integration to a compact set on which we have have good control of the size of $||u||_{S^{3+d}}$ and $|\det Q_\phi(y_0)|^{-1/2}$.  In our application, the amplitude function $u$ will actually depend on the parameter $\lambda$.  So the version of Lemma \ref{statphase} that we shall use will be of the following form.  Let $u(y,\lambda)\in C^\infty(\R^{d+1})$.  Suppose that 
\begin{enumerate}
\item\label{1cond} 
for all $\lambda\geq 1$ the support of $u(\cdot,\lambda)$ is contained in a fixed compact set $K\subset\R^d$;
\item\label{2cond} $||u(\cdot ,\lambda)||_{S^{3+d}}\ll 1$ uniformly in $\lambda$;
\item\label{3cond} $|\det Q(\phi(y_0))| \gg  1$.
\end{enumerate}
Then it follows immediately from Lemma \ref{statphase} that
\begin{equation}\label{big O}
\int_K e^{i \lambda \phi(y)} u(y,\lambda) dy=O(\lambda^{-d/2}).
\end{equation}

We now proceed to the proof of \eqref{intermediate} for the real and complex places.  The estimate \eqref{big O} will be used with $d=1$ and $d=2$, respectively.

\bigskip

\noindent {\bf Proof of \eqref{intermediate} if $v = \Bbb{R}$.}  We recall the expression \eqref{g star sum}.  We will only estimate the integral $I(x,\delta)$ for $\delta=0$, the case $\delta=1$ being similar.  We therefore write $I(x)=I(x,0)$ for notational simplicity.

We begin by isolating the region around the real axis where the quotient of Gamma functions fails to oscillate.  Let $\omega_0\in C^\infty_c(\R)$ be such that $\omega_0(\tau)=1$ on $|\tau|\leq M+1$ and $\omega_0(\tau)=0$ on $|\tau|\geq M+2$, where $M=\max|\Im\mu_j|$.  Put $\omega=1-\omega_0$.  Then 
\begin{equation*}
\int_\R \omega_0(\tau)\widehat{g_v}(-i\tau) \gamma(-i\tau,\pi_v\times\widetilde{\pi}_v,\psi_v)|x|^{-1-i\tau}d\tau=O(1/T|x|).
\end{equation*}
On the remaining integral, we can now insert the precise version of Stirling's formula (with oscillating factor) in the form of
\begin{displaymath}
 \gamma(-i\tau,\pi_v\times\widetilde{\pi}_v,\psi_v)=\alpha(\tau)|\tau|^{n^2/2}(|\tau|/2\pi e)^{in^2\tau},
\end{displaymath}
for $|\tau| \geq M+1$, where $\alpha(\tau)$ is a smooth bounded function satisfying $\alpha^{(j)}(\tau) \ll |\tau|^{-j}$ for integers $j\geq 0$.  Putting $f(\tau)=\widehat{g}_v(-i\tau) \alpha(\tau)\omega(\tau)$, we have $f^{(j)}(\tau) \ll_{j, A} T^{-1} |\tau|^{-j} (1+|\tau|/T)^{-A}$ by \eqref{g*inf}, and
\begin{displaymath}
I(x)=|x|^{-1}\int_\R f(\tau)|\tau|^{n^2/2}\exp\left(i\tau\left(n^2\log \frac{|\tau|}{2\pi e} -\log|x|\right)\right) d\tau+O\left(\frac{1}{T|x|}\right).
\end{displaymath} 
By a change of variables $\tau\mapsto 2\pi |x|^{1/n^2} \tau$ we obtain
\begin{equation}\label{integral}
I(x)= |x|^{-\frac{1}{2}+ \frac{1}{n^2}} \int_\R \tilde{f}(\tau,|x|) \exp(i |x|^{1/n^2} \phi(\tau)) d\tau+O(1/T|x|),
\end{equation}
where $\tilde{f}(\tau,|x|)=(2\pi)^{1+n^2/2} f(2\pi |x|^{1/n^2} \tau)|\tau|^{n^2/2}$ is supported on $|\tau| \gg |x|^{-1/n^2}$  and
\begin{equation*}
\phi(\tau) = 2 \pi  n^2 \tau (\log |\tau| - 1).
\end{equation*}
The phase function $\phi$ has a stationary point at $\tau_0=1$ and nowhere else.  This critical point is non-degenerate.

We apply a smooth partition of unity and cut out smoothly the interval $[1/2, 2]$ in \eqref{integral}. The remaining portion of \eqref{integral} is, by repeated integration by parts, $O(T^{-1} |x|^{-1/2+ \varepsilon})$, which is most easily seen by splitting the range of integration into dyadic intervals.   Hence \eqref{integral} equals
\begin{displaymath}
   |x|^{-\frac{1}{2} + \frac{1}{n^2}} \int_{1/4}^{3} u(\tau, |x|) \exp(i |x|^{1/n^2} \phi(\tau)) d\tau + O(T^{-1}|x|^{-1/2+\varepsilon})
\end{displaymath}
where $u(\cdot , |x|)$ has compact support in $(1/4, 3)$ on which we have the uniform size condition  $u^{(j)}(t, |x|) \ll_j T^{-1}$ for all $j \in \Bbb{N}_0$ and $\phi^{(j)}(t) \asymp_j 1 $ for $ j \geq 2$. Therefore we can apply \eqref{big O} in the case $d=1$  and bound the main term of \eqref{integral} by $O(T^{-1}|x|^{-\frac{1}{2} + \frac{1}{2n^2}})$.

\medskip

\noindent {\bf Proof of \eqref{intermediate} if $v = \Bbb{C}$.}  The complex case is very similar to the real case, but a few extra ingredients are necessary. We will highlight the main points, leaving the rest of the argument for the reader to fill in.

Using the notation of \eqref{g star sum} and \eqref{ga} and writing $\theta=\arg(z)\in [-\pi,\pi)$, we have
\begin{equation*}
  g_v^*(z)=\frac{1}{2\pi} \sum_{\delta\in \Z}e^{-i\delta\theta}I(z, \delta).
\end{equation*}
As in the real case we introduce   a smooth weight function $\omega$ to restrict the support of $\delta$ and $\tau = \Im s$ in the integral \eqref{ga} such that   $|\delta| \gg 1$ with an error of $O(T^{-1} |x|^{-\frac{1}{2} + \frac{1}{2n^2}})$.  
Combining positive and negative $\delta$, we obtain after  Poisson summation 
  \begin{equation}\label{complete}
  g_v^*(z)=   \sum_{\delta\in \Z} \widetilde{I}_{+}(z,\delta) +  \widetilde{I}_{-}(z,\delta),
  \end{equation}
where
\begin{equation*}
 \widetilde{I}_{\pm}(z,\delta)=\int_{-\infty}^{\infty} \int_0^{\infty} \omega(s) e^{\pm i\sigma(\delta-\theta)}\widehat{g_v}(-i\tau, -\sigma)\gamma(-i\tau ,\pi_v\otimes \sigma \times \widetilde{\pi}_v,\psi_v)|z|^{-2(1+i\tau)}d\sigma \, d\tau.
\end{equation*}
In the above integral we have written $s=\sigma+it$.  By  \eqref{infinite gamma}  we have
\begin{displaymath}
  \gamma(s, \pi_v \otimes \delta \times \tilde{\pi}_v,\psi_v) = \epsilon (s,\pi_v\times\widetilde{\pi}_v,\psi_v)i^{\delta}\prod_{j=1}^{n^2} \frac{\Gamma_\C(1-s - \mu_{\pi \times \tilde{\pi}}(v, j) + \delta/2)}{ \Gamma_\C(s - \mu_{\pi \times \tilde{\pi}}(v, j) + \delta/2)}
\end{displaymath}
for $\delta > 0$, and the latter expression makes perfect sense for $\delta = \sigma \in \Bbb{R}$. 

It is now easy to see that the $\delta$-sum in \eqref{complete} is rapidly converging, and $O_\varepsilon ((T|z|)^\varepsilon)$ of $\delta$ contribute non-negligibly to $g_v^*(z)$.  It suffices then to estimate $\widetilde{I}_{\pm}(z,\delta)$ for a single $\delta$.  For convenience of exposition we bound $\widetilde{I}(z) := \widetilde{I}_{+}(z, 0)$, all other cases are essentially identical.   Stirling's formula then reads 
\begin{displaymath}
\gamma(-i\tau ,\pi_v\otimes (2\sigma) \times\widetilde{\pi}_v,\psi_v)= \alpha(s)|s|^{n^2}\exp \left(2i n^2\left(\tau \log \frac{|s|}{2\pi e} +\sigma\arctan\frac{\tau}{\sigma}\right)\right)
\end{displaymath}
for $s\in \text{supp}(\omega)$, where $\alpha(s)$ is a smooth bounded function satisfying $\alpha^{(i, j)}(s) \ll_{\pi_v,i,j} |s|^{-i+j}$ for integers $i, j\geq 0$ and $|s|\geq M+1$. We must consider then the integral
\begin{equation*}
|z|^{-2}   \int_{-\infty}^{\infty}\int_0^{\infty} f(s)|s|^{n^2}\exp \left(2i n^2\left(\tau \log \frac{|s|}{2\pi e} +\sigma \left(\frac{\pi}{2}+ \arctan\frac{\tau}{\sigma}\right)\right)-2i\tau\log |z|-2i\sigma\theta\right)d\sigma \, d\tau,
  \end{equation*}
where $f(s)=\widehat{g_v}(-i\tau, -\sigma)\alpha(s)\omega(s)$.  Changing variables $s\mapsto 2 \pi |z|^{1/n^2}s$ we obtain
\begin{equation*}
|z|^{-1+\frac{2}{n^2}}   \int_{\R^2} \tilde{f}(s,|z|) \exp(i |z|^{1/n^2} \phi(s)) ds,
  \end{equation*}
where $\tilde{f}(s,|z|)=(2\pi)^2 f(|z|^{1/n^2}s)|s|^{n^2}$ is supported on $|\tau| \gg |z|^{-1/n^2}$ and
\begin{displaymath}
\phi(s) = 4\pi n^2 \left(\tau  (\log |s|-1) +\sigma\left(\frac{\pi}{2} - \frac{\theta}{n^2}+\arctan\frac{\tau}{\sigma}\right)\right).
\end{displaymath}

Now let us examine the critical points of $\phi$.  For $s \in \text{supp}(\omega)$  we have
\begin{displaymath}
  \nabla \phi(s) = 4\pi n^2 \begin{pmatrix} \arctan(\tau/\sigma) + \pi/2 - \theta/n^2 \\ \log |s|  \end{pmatrix}.
\end{displaymath}
There is at most one stationary point $\nabla\phi(s)=0$. 
Any such critical point is non-degenerate since the Hessian
\begin{displaymath}
Q_\phi(s) = \frac{4\pi n^2}{|s|^2}\begin{pmatrix} -\tau & \sigma\\ \sigma & \tau\end{pmatrix}
\end{displaymath}
has determinant $-4\pi n^2$ along $|s|=1$.

Having identified the critical points and established their non-degeneracy, the rest of the argument (partition of unity, integration by parts away from the critical point giving an error $O(T^{-1} |x|^{-\frac{1}{2} + \frac{1}{2n^2} + \varepsilon})$, application of \eqref{big O} with $d=2$ around the critical point) now follows that of the real case.
 
\subsection{$S$-adic estimate}

We now average the local estimates in Proposition \ref{local lemma} to obtain our next Corollary.  Recall the definition \eqref{capital Gs} of $G_S$. Henceforth we use the abbreviation $V$ for the volume
\begin{equation*}
V=V(\cc,D,T)\asymp \mathcal{N}(\cc)T^{r-1}.
\end{equation*} 

\begin{cor}\label{G star} For $|x|_S \geq 1$ we have
 \begin{equation}\label{GS}
  G_S^{\ast}(x) \ll_{A, \varepsilon, \pi, |S|} V^{\frac{n^2-1}{2}+\varepsilon}   |x|_S^{-1}  \Bigl(1+ \Bigl(\frac{|x|_S}{V^{n^2+\varepsilon}}\Bigr)^{-A}\Bigr)
  \end{equation} 
  for any $A\geq 0$, $\varepsilon > 0$. 
\end{cor}

Before proving the corollary we indicate the approach.  The function $g^{*}_S$ has essential support inside a box of volume about $V$. By Dirichlet's unit theorem, the $S$-units are logarithmically distributed in $K_S^{\times}$, and hence only $V^{\varepsilon}$ terms contribute in a non-negligible way.  This heuristic, at least for those units in $\mathcal{O}_K^\times$, has been made precise in a useful lemma of Bruggeman-Miatello \cite{BM}.

\begin{lemma}[Bruggeman-Miatello]\label{BM lemma}  Let $g : K_{\infty}^\times \rightarrow \Bbb{C}$ be a function satisfying $|g(x)| \leq \prod_{v \mid \infty} \min(1, |x_v|_v^{-A})$ for some $A\geq 0$. Then
\begin{displaymath}
  \sum_{u \in \mathcal{O}_K^\times} |g(ux)| \ll_A  \min (1+|\log |x|_{\infty}|^{r-1}, |x|_{\infty}^{-A}).
\end{displaymath}
\end{lemma}

We now proceed to the proof of Corollary \ref{G star}.

\begin{proof} For $v \in S$ put
\begin{equation*}
{\rm vol}(g_v^*)= \begin{cases} T^{n^2+\varepsilon}, & \text{if } v\mid \infty,\;  v\not= v_0;\\
                             \mathcal{N}\p^{e_{\p}(n^2+\varepsilon)},& \text{if } v  = \p \text{ is finite};\\
                             1,& \text{if } v=v_0.
                             \end{cases}
\end{equation*}
Here we agree to set $e_{\p} = 0$ if $\p \in B_{\pi,K}$.  Let
\begin{equation*}
{\rm vol}(g_\infty^*)=\prod_{v \mid \infty} {\rm vol}(g_v^*)=T^{(r-1)(n^2+\varepsilon)}\quad\text{and}\quad {\rm vol}(g_S^*)=\prod_{v \in S} {\rm vol}(g_v^*) = V^{n^2+\varepsilon}.
\end{equation*}
The estimates of Proposition \ref{local lemma} yield
\begin{equation*}
g_S^{\ast}(x) \ll |x|_S^{-1 } V^{n^2(\frac{1}{2} - \frac{1}{2n^2}) +\varepsilon} M(x),
\end{equation*}
where
\begin{displaymath}
M(x)= \prod_{v \in S}M_v(x_v)\quad\text{and}\quad  M_v(x_v)= \left(1 + \frac{|x_v|_{v}}{{\rm vol}(g_v^*)}\right)^{-A}.
\end{displaymath}
We proceed  to estimate $\sum_{u\in \mathcal{O}_S^{\times}}M(ux)$. We index the sum by first fixing a set $\{u' \}$ of representatives of  $\mathcal{O}_S^\times/\mathcal{O}_K^\times$, and then summing over $uu'$ for $u\in\mathcal{O}_K^\times$.  By Lemma \ref{BM lemma} we have
\begin{displaymath}
\begin{split}
\sum_{u\in\mathcal{O}_K^\times}M_\infty(uu'x_\infty)\ll\min\Bigl( 1+ \Bigl|\log \frac{|u'x_\infty|_{\infty}}{{\rm vol}(g_\infty^*)}\Bigr|^{r-1}, \Bigl(\frac{|u'x_\infty|_{\infty}}{{\rm vol}(g_\infty^*)}\Bigr)^{-A}\Bigr). 
 \end{split}
\end{displaymath}
If $u'\in \mathcal{O}_S^{\times}/\mathcal{O}_K^\times$ has $\p$-valuation $k_\p$, say, then letting $\ell_{\p} := k_\p+e_\p n^2+v_\p(x_\p)$ we have
\begin{equation*}
 M_\p(u_\p x_\p)=(1 + \mathcal{N}\p^{-\ell_{\p}})^{-A}. 
\end{equation*}
We can now view the sum over $u'$ as a sum over certain integer vectors $\ell = (\ell_{\p})_{\p \in S_{\text{fin}}} \in \Bbb{Z}^{|S_{\text{fin}}|}$, where $S_{\text{fin}}$ denotes the set of finite primes in $S$.  If $u'$ corresponds to $\ell$ then
\begin{equation*}
\frac{|u'x_\infty|_{\infty}}{{\rm vol}(g_\infty^*)}= \frac{|x|_S}{V^{n^2+\varepsilon}}\prod_{\p\in S_{\text{fin}}}\mathcal{N}\p^{\ell_\p}:=X_\ell.
\end{equation*}
Thus we find
\begin{align*}
  \sum_{u \in \mathcal{O}_S^{\times}} M(ux) &\ll \sum_{\ell \in \Bbb{Z}^{|S_{\text{fin}}|}} \min\Bigl( 1+ |\log X_{\ell} |^{r-1},X_{\ell}^{-A}\Bigr)\prod_{\p \in S_{\text{fin}}} (1 + \mathcal{N}\p^{-\ell_{\p}})^{-A} \\
&\ll \min\Big(1+\left| \log \big(|x|_S/V^{n^2+\varepsilon}\big)\right|^{|S|-1}, \big(|x|_S/V^{n^2+\varepsilon}\big)^{-A}\Big),
\end{align*}
as one confirms easily by induction. This implies the Corollary.\end{proof}

\section{Proof of Theorem \ref{non-van thm}}\label{non-van section}

Let $0 < Y < 1$ be a parameter to be chosen later, and let $x \in \Bbb{I}$ be the idele satisfying 
\begin{equation}\label{defx}
 x_{v_0} = Y^{1/[K_{v_0}:\Bbb{R}]}; \quad  x_v = 1, \quad v \not= v_0.
\end{equation}
We choose ramification data as in Theorem \eqref{non-van thm} and as before abbreviate $V = V(\cc, D, T)$. We apply the summation formula in Proposition \ref{Voronoi} with the test function defined in Section \ref{sumform} and $x$ as in \eqref{defx}.  
By the positivity of the coefficients $\lambda_{\pi\times\widetilde\pi}(\a)$ and the functions $g_v$ we may drop all but the term corresponding to $\mathfrak{a} = (1)$  and $u=1$ in the sum \eqref{nothard1}  getting
\begin{equation}\label{bound1}
  G(x) \geq   g_{v_0}(Y^{1/[K_{v_0}:\Bbb{R}]}) \gg  Y^{-\beta}. 
\end{equation}
Next, using \eqref{GS} and \eqref{nothard2} together with the fact that $|x|_{\A} = |x|_S$ we estimate
\begin{equation*}
   |x|_\A^{-1}G^{\ast}(1/x) \ll V^{\frac{n^2-1}{2}+\varepsilon}   \sum_{\mathfrak{a}\subseteq\mathcal{O}_K} \frac{\lambda_{\pi \times \tilde{\pi}}(\mathfrak{a})}{\mathcal{N}(\a)} \left(1 + \frac{\mathcal{N}(\a)}{YV^{n^2+\varepsilon}}\right)^{-1}.
\end{equation*}
The absolute convergence of the Rankin-Selberg $L$-function to the right of $\Re s = 1$ implies
\begin{equation}\label{bound1a}
|x|_\A^{-1}G^{\ast}(1/x)\ll V^{\frac{n^2-1}{2}+\varepsilon}.
\end{equation}
Finally we come to the residual terms in Proposition \ref{Voronoi}.  It is easy to see that
\begin{equation}\label{bound2}
|x|_\A^{-1}(\underset{s=0}{\rm Res}+\underset{s=1}{\rm Res})\ll_\pi \widehat{g}_S(0, \text{triv})+|x|_\A^{-1}\widehat{g}_S(1, \text{triv})\asymp (YV)^{-1}.
\end{equation}
The remaining poles come exclusively from $\widehat{g}_{v_0}(s, \omega)$.   We first write $\omega = \prod_v \omega_v \in \widehat{\mathscr{C}^1(\cc)}$ where the component of $\omega_{v_0}$ at $v_0$ is $(i\tau_0, \delta_0)$, say.  Then  $\widehat{g_{v_0}}(s)$ has a pole at $s = \beta - i\tau_0$ if and only if $\delta_0 = 0$.  Thus  the  poles strictly within the critical strip contribute
\begin{equation}\label{before}
\begin{split}
& c_K \sum_{\substack{\omega \in \widehat{\mathscr{C}^1(\cc)}\\ \omega_{v_0} = (i\tau_0, 0)}} \bar{\omega}(x)   L^S(\beta - i\tau_0, \pi\otimes\omega\times\widetilde{\pi})\widehat{g}_{S\setminus \{v_0\}}(\beta-i\tau_0, \omega)|x|_{\Bbb{A}}^{\beta - i\tau_0}\\
& =  c_K Y^{-\beta}  \sum_{\substack{\chi \in \widehat{\mathscr{C}(\cc)}\\ \tau_{v_0}(\chi)=0}}  \bar{\chi}(x) L^S(\beta,  \pi\otimes\chi \times \widetilde{\pi}) \prod_{v \in S\setminus \{v_0\}} \widehat{g}_v(\beta, \chi)\\
& \ll_A (Y^{\beta}V)^{-1} \sum_{\substack{\chi \in \widehat{\mathscr{C}(\cc)}\\ \tau_{v_0}(\chi)=0}}   |L^S(\beta,  \pi\otimes\chi \times \widetilde{\pi})|\prod_{\substack{v\mid\infty\\ v\neq v_0}} \left(1 + \frac{\mathcal{C}(\chi_v)}{T}\right)^{-A}\\ 
& \ll_{\varepsilon} (Y^{\beta}V)^{-1} \sum_{\chi \in X(\cc, D, TV^{\varepsilon})}   |L^S(\beta,  \pi\otimes\chi \times \widetilde{\pi})| + V^{-100}.  
\end{split}
\end{equation}
Here we used \eqref{hatgfin}, \eqref{g*inf} and the definition \eqref{def of X}, and in the final step truncated the series at the cost of a negligible error.  Combining \eqref{bound1} -- \eqref{before}, we obtain
\begin{equation}\label{final}
\begin{split}
 &  \sum_{\chi \in X(\cc, D, TV^{\varepsilon}) }| L^S(\beta, \pi\otimes\chi \times \widetilde{\pi})  |    \gg V + O_\varepsilon (Y^{\beta-1} + V^{\frac{n^2+1}{2}+\varepsilon} Y^{\beta}). 
\end{split}  
\end{equation}
We choose $Y= V^{-\frac{n^2+1}{2}}$.  Then for $\beta > 1 - 2/(n^2+1)$ and $\varepsilon$ sufficiently small, the main term on the right hand side of \eqref{final} dominates the error term.  This completes the proof of Theorem \ref{non-van thm}.

\bibliographystyle{plain}

\printindex

\end{document}